\documentclass[onecolumn,12pt]{article}
 
\usepackage{amssymb,amsmath,amsthm,calc,adjustbox}
\usepackage{authblk}

\usepackage{accents} 
\usepackage{url}

\usepackage{dsfont}
\usepackage{adjustbox,fullpage}

\usepackage[toc,page]{appendix}

\usepackage{acronym}
\acrodef{PPP}[PPP]{Poisson Point Process}
\acrodef{NPPP}[NPPP]{Non-homogeneous PPP}
\acrodef{CDF}[CDF]{Cumulative Distribution Function}
\acrodef{PDF}[PDF]{Probability Distribution Function}
\acrodef{PMF}[PMF]{Probability Mass Function}
\acrodef{PCF}[PCF]{Pair Correlation Function}
\acrodef{RV}[RV]{Random Variable}
\acrodef{i.i.d.}[i.i.d.]{independent and identically distributed}
\acrodef{w.r.t.}[w.r.t.]{with respect to}
\acrodef{V2V}[V2V]{Vehicle-to-Vehicle}
\acrodef{1D}[1D]{one-dimensional}
\acrodef{2d}[2d]{two-dimensional}
\acrodef{VANET}[VANET]{Vehicular ad hoc network}

\usepackage{tikz}
\usepackage{cite}
\usetikzlibrary{automata,topaths}
\usetikzlibrary{shapes}
\usetikzlibrary{plotmarks}

\usepackage{braket}
\usepackage{subcaption}

\usepackage{listings,xcolor}

\usepackage{wasysym} 
\usepackage{float} 

\usepackage{color} 
\usepackage{hyperref}
\hypersetup{
   colorlinks=true, %set true if you want colored links
    linktoc=all,     %set to all if you want both sections and subsections linked
    linkcolor=blue,  %choose some color if you want links to stand out
    citecolor=blue,
}

\usepackage{enumerate}

\newtheorem{theorem}{Theorem}
\newtheorem{corollary}{Corollary}[theorem]

\newtheorem{definition}{Definition}[section]% theorem counter resets every \subsection
\numberwithin{figure}{section}
\numberwithin{equation}{section}

\def\E{\mathop{\hbox{\rm I\kern-0.20em E}}\nolimits}

\makeatother

\allowdisplaybreaks

\begin{document}

\title{
  \huge
 Connectivity of 1d random geometric graphs %1d random geometric graphs, lattice paths, and spatially random permutations % Lattice paths, Ferrers diagrams, and the connectivity of 1d random geometric graphs
}

\author[1]{Alexander P. Kartun-Giles\thanks{alexander.giles@ntu.edu.sg}}
\author[2]{Kostas Koufos \thanks{konstantinos.koufos@gmail.com}}
\author[1]{Nicolas Privault\thanks{nprivault@ntu.edu.sg}}
%\affiliation{Division of Mathematical Sciences \\School of Physical and Mathematical Sciences \\Nanyang Technological University \\21 Nanyang Link \\Singapore 637371}
%\pacs{89.75.Hc,89.75.-k,89.75.Fb}
\affil[1]{
   \small
   Division of Mathematical Sciences, School of Physics and Mathematical Sciences, Nanyang Technological University, Singapore, 21 Nanyang Link, Singapore 637371}
\affil[1]{
   \small
   School of Mathematics, University of Bristol, Fry Building, Woodland Road, Bristol, BS8 1UG, United Kingdom}

\maketitle

\baselineskip0.6cm

\vspace{-0.6cm}

\begin{abstract}
A \textit{1d random geometric graph} (1d RGG) is built by joining a random sample of $n$ points from an interval of the real line with probability $p$. We count the number of $k$-hop paths between two vertices of the graph in the case where the space is the 1d interval $[0,1]$.  We show how the $k$-hop path count between two vertices at Euclidean distance $|x-y|$ is in bijection with the volume enclosed by a uniformly random $d$-dimensional lattice path joining the corners of a $(k-1)$-dimensional hyperrectangular lattice. We are able to provide the probability generating function and distribution of this $k$-hop path count as a sum over lattice paths, incorporating the idea of restricted integer partitions with limited number of parts. We therefore demonstrate and describe an important link between spatial random graphs, and lattice path combinatorics, where the $d$-dimensional lattice paths correspond to spatial permutations of the geometric points on the line.

  %Usually, $p$ is a non-increasing function of the Euclidean distance between the points, so nearby points are more likely to be connected, though long range links are allowed in some models. They often appear as models of wireless networks, in topological data analysis, and in models of theoretical physics, such as complex networks, or as causal sets in spacetime.
  
% This article concerns the combinatorics of spatially random permutations of a countable set of points on a line, and their relation to the connectivity of 1d random geometric graphs. These 1d random graphs are geometric networks formed by connecting pairs of points with some probability that is a function $H$ of their mutual Euclidean separation. We provide, using the Gaussian binomial coefficient and a high-dimensional analogue, an exact expression for the probability generating function of the $k$-hop path count between two vertices at Euclidean separation $|x-y|>0$ in terms of the vertex density and $r_0$. The proof relies on detailing a bijection between $d$-dimensional lattice paths which fit inside a box, and the number of $k$-hop paths between two vertices of the graph. We also detail a moment generating function, and the probability of zero paths, with extensions in the case of 2d random geometric graphs.
\end{abstract} 

\noindent
{\em Keywords}: Random geometric graphs,
unit disk model,
connectivity,
$k$-hop paths,
Ferrers diagrams,
lattice paths,
Gaussian binomial coefficients. 

\noindent
    {\em Mathematics Subject Classification (2020): 05C80, % Random graphs (graph-theoretic aspects)
05C40, % Connectivity
11B65. % Binomial coefficients; factorials; $q$-identities
    }

\baselineskip0.7cm

\parskip-0.1cm

\newpage

\tableofcontents

\section{Introduction}
Perhaps the simplest example of a random spatial network  \cite{barthelemy2018,kartungiles2019,georgiou2013,georgiou2015,georgiou2016,kartungiles2016,mulder2018,cunningham2017,fountoulakis2020} is the 1d random geometric graph (1d RGG) \cite{wilsher2020,kartungiles2020,penrosebook,drory1997,knight2017,grimmettbook,boguna2020}. 1d problems are widely studied in order to first understand a simpler case, such as the Ising and Heisenberg models of magnetism studied by Ising and Hans Bethe in 1924 and 1931 respectively, in the later case resulting in the famous Bethe ansatz, as well as more modern examples of e.g. 1d statistical mechanics of nucleosome positioning on genomic DNA, or 1d stochastic traffic flow models \cite{lieb1967,liebreview,tesoro2016,pesheva1997}.

The 1d case of the random geometric graph has appeared in three major places. Firstly, as random spatial models in the physics of complex systems, see for example the 1d soft random geometric graph \cite{wilsher2020} used in complex networks by Krioukov et al. in network geometry \cite{krioukov2016}. Secondly, in Poisson-Boolean continuum percolation \cite{drory1997,meester1996,shalitin1981,gori2017}, and thirdly, in vehicular communications \cite{foh2004,foh2005,han2007,wilsher2020,kartungiles2020,knight2017,shang2009,ng2011,zhang2012,zhang2014,mao2017,ajeer2011,gupta1999,mao2017,koufos2016,koufos2018,koufos2019,kartungiles20182,kartungiles20202}. For 1d spatial models similar to the 1d RGG, see the 1d exponential random geometric graph \cite{gupta2008}, random interval graphs where the connection is between overlapping intervals of random length \cite{godehardt1996,shang2009}, or various models of one-dimensional mathematical physics \cite{drory1997, lieb1967,liebreview}. For a historical introduction to the similar problem of covering a line by random overlapping intervals, see Domb \cite{domb1989}.

1d RGGs do not undergo the connectivity phase transition in the same, non-trivial way as their 2d counterparts \cite{lieb1967}.  This has the advantage of simplifying the underlying geometric probability. M. D. Penrose specifically points out that the connectivity transition in \cite[Theorem 6.1]{krivelevich2016}, ``is different (when $d=1$) because 1-space is less connected''. The connectivity transition in 1d has in fact been been solved by A. Drory in the hard case, via a statistical physics approach using the Potts model \cite{drory1997}. Study of connectivity in the soft case can be found in Wilsher et al. \cite{wilsher2020}, where hard and soft connectivity in the 1d case are compared, and connectivity in the soft case is studied. The soft case \cite{knight2017,kartungiles2020,ng2011,zhang2012,zhang2014,mao2017} remains unsolved.

The value of understanding the mathematics of the 1d case is well motivated based on some recent problems in spatial complex networks \cite{kartungiles2015,kartungiles2019,kartungiles2018,kartungiles20183,kartungiles2016,kartungiles20162,knight2017,privault2019}, particularly betweenness centrality \cite{kartungiles2015}, which involves counting the number of paths between two nodes in a complex network \cite{kartungiles2018,kartungiles2019}. Therefore, in this article, we focus on counting the integer number $\sigma_{k}$ of $k$-hop paths which run between two fixed vertices of the 1d RGG. The problem links random geometric graphs to the problem of counting multidimensional lattice paths on the $d$-dimensional integer lattice.

The main point of interest of this article is a connection between paths in random geometric graphs, and the volume beneath a multidimensional lattice path. This also links the problem to the theory of integer partitions, since the volume under the path is a restricted integer partition with a limited number of parts. We put the restriction Eq.~\eqref{e:range1} on the connection range $r_0$ of the vertices in the random geometric graph, to avoid an important complication known as ``overlapping lenses'', see Section \ref{sec:overlap}. This may at a later stage be overcome by consisdering a more sophisticated lattice path counting problem involving downward steps rather than just up, right, and so on, working toward the terminal point at the corner of the lattice,  but the link to lattice path enumeration remains, which is the focus of the article. We leave relaxing Eq.~\eqref{e:range1} it to a later study.

Our results are expressions for the probability mass function (p.m.f.) and probability generating function (p.g.f) for the number $\sigma_{k}$ of $k$-hop paths between two vertices of the 1d RGG, in terms of the point process density $\lambda$, and the connection range $r_0$. The focus is on the role lattice paths, spatially random permutations, and random integer partitions play. A closely related article is Janson \cite[Section 3]{janson2012}, where the distribution of the area under a multidimensional lattice path is considered as a U-statistic, though this is not related to 1d RGGs.

This paper is organised as follows. In Section \ref{sec:summary} we summarise our results on 1d RGGs. In Section \ref{sec:preliminaries}, we introduce some preliminary ideas about lattice paths and Ferrers diagrams that appear throughout the paper, and define our notation. In Section \ref{sec:allres}, we provide proofs of an expression in terms of a sum over lattice paths for the p.m.f. of the number of $k$-hop paths between two nodes, and provide proof of a p.g.f for this quantity, as is typical in enumeration problems. We also include some Mathematica code for the problem in Appendix \ref{a:1}.

%\subsection{Background}

% The number of vertices however cannot be controlled with the level of sophistication we require, specifically, since $m_1,\dots,m_{k-1}$, which we consider to be vertices in various $k$-hop ranges of the boundary nodes, must all be Poisson variates with a connection range-dependent paramater given by Eq. \eqref{e:width}. Nevertheless, this is an important article which considers the rare case of extending the the 2d area under a random lattice path to a higher-dimensional case. Betz and Ueltschi study spatially random permutations of a countable set of points on a line in \cite{betz2008}, among other research in statistical physics, though these are not random graphs. Pattern avoiding permutations \cite{krattenthaler2001} are also relevent to the connectivity, though we do not study this here.

\section{Summary of results}\label{sec:summary}

 \begin{figure}
\centering
\includegraphics[scale=0.265]{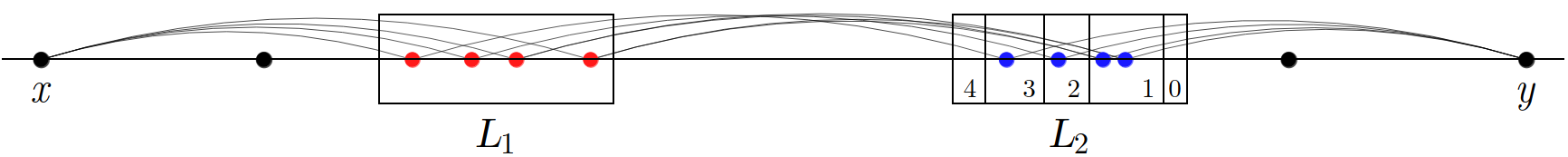}
\caption{The two lenses $L_1$ and $L_2$, and their nodes (red and blue respectively). The nodes in the left lens $L_1$ connect to the nodes in the right lens $L_2$ when they are within a range $r_0$. This is depicted with the red bars in $L_2$, which are bins which give path counts via each nodes. The three-hop path count in this case is $\sigma_{3}=1+1+2+3=7$.}\label{fig:3hops} \vspace{5mm}
\includegraphics[scale=0.265]{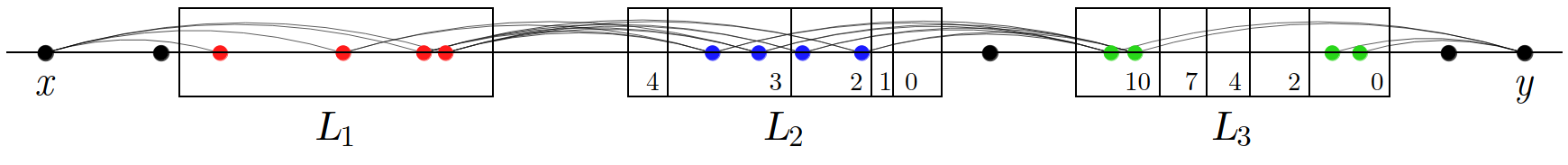}

\caption{The three lenses $L_1,L_2$ and $L_3$, and their nodes (red, blue and green, respectively). The nodes in the left lens $L_1$ connect to the nodes in the right lens $L_2$ when they are within a range $r_0$, and similarly for $L_2$ and $L_3$. This is depicted with the red bars in $L_2$ and the blue bars in $L_3$. The value of $\sigma_{4}$, in this case, is $\sigma_{4}=0+0+10+10=20$.}\label{fig:3hops2}
\includegraphics[scale=0.265]{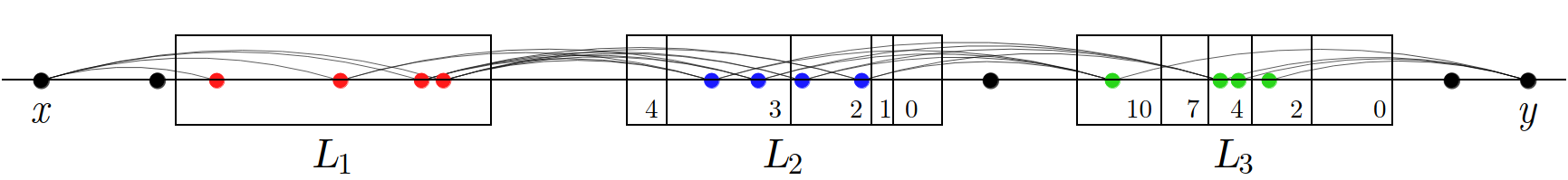}

\caption{In this case, is $\sigma_{4}=2+4+4+10=20$. The six lattice paths which give this exact partition of 20 are depicted in Fig. \ref{fig:main-22}. Since there are six, the partition degeneracy of $2+4+4+10 \vdash 20$ is $6$.}\label{fig:used}
\includegraphics[scale=0.265]{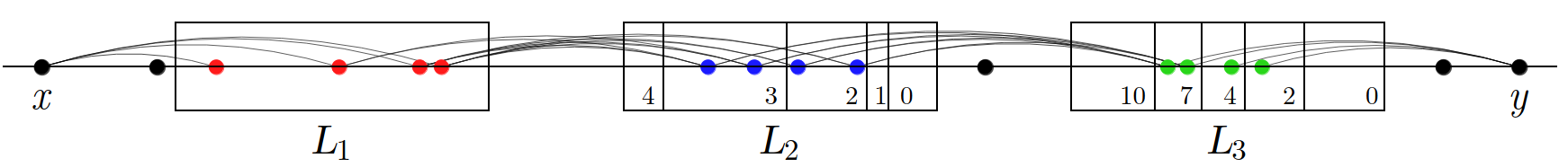}

\caption{In this case, is $\sigma_{4}=2+4+7+7=20$. The partition degeneracy in this case is 2. We add the partition degeneracy of all partitions to get the total number of lattice paths to which it corresponds.}
\label{fig:k4-1}
\end{figure}

\subsection{Introduction and notation}
Consider the 1d random geometric graph $G$ defined in Section \ref{sec:1drgg}. We are able to state the probability mass function of the number of $k$-hop paths in $G$ which run between the two vertices conditioned to exist at the endpoints of the domain $[0,1]$ as a sum over lattice paths, in the following way. Consider the $d$-dimensional hyperrectangle intersecting $\mathbb{Z}^{d}$, which is the hyperrectangular lattice
\begin{equation}\label{e:rec}
\mathcal{M}_{d}:=  \{0,\dots,m_1\} \times \dots \times \{0,\dots,m_{d}\}, 
\end{equation}
 Given two lattice points $A,\Omega \in \mathcal{M}_{d}$, a set of steps $\mathbb{S}$ (i.e. unit vector movements between adjacent lattice sites), and an integer $m>0$, we denote by $L_{m}(A \to \Omega ;\mathbb{S})$ the set of lattice paths from $A$ to $\Omega$ with $m$ steps in the set $\mathbb{S}$. Then with $e_i$ denoting the vector with a 1 in the $i^{\text{th}}$ position, and 0 elsewhere, we consider
\begin{equation}
\mathcal{P}([m_{k_{1}} \times \dots \times m_{k_{n}}]) := L_{m}(\mathbf{0} \to  (m_{k_{1}},\dots,m_{k_{n}}) ; \{e_i: 1 \leq i \leq n \} )
\end{equation}
to be the set of lattice paths between $\mathbf{0}$ and $(m_{k_{1}},\dots,m_{k_{n}})$ in the lattice $[m_{k_{1}} \times \dots \times m_{k_{n}}]$. Then, consider a lattice path $\Pi_{k-1} \in \mathcal{P}(\mathcal{M}_{k-1})$. We need six further straightforward ingredients derived from $\Pi_{k}$. Firstly, as shown in Fig.~\ref{fig:project}, the projections $\Pi_{ij}(\Pi_{k-1})$ of $\Pi_{k-1}$ onto the faces of $\mathcal{M}_{k-1}$, which are ${m_i+m_j \choose m_i}$-letter words composed of two types of letter. Secondly, the integer partition $\pi_{ij}\left(\Pi_{k-1}\right)$ corresponding to the Ferrers diagram under the projection $\Pi_{ij}(\Pi_{k-1})$, which is a sequence of $m_i$ integers each of size no larger than $m_j$, and whose $t^{\text{th}}$ part is $\pi_{ij}(t)$, where the context is clear that we are considering the $t^{\text{th}}$ part of the integer sequence $\pi_{ij}\left(\Pi_{k-1}\right)$. When the subscripts are omitted in the case $\pi(\Pi)$, this denotes the integer partition corresponding to the 2d lattice path $\Pi$, though the context will be clear.

Thirdly, since the hypervolume $V$ under the path may be interpreted as a restricted integer partition of $V$ using exactly $m_{k-1}$ parts from the set $S(\Pi_{k-2})$, we need the integer sequences of Eq.~\eqref{e:ss1}, which we detail there.  We write
\begin{equation}
 S(\Pi_{k-1}) = \left(S_{i}(\Pi_{k-1})\right)_{i=0,\dots,m_{k-2}}
\end{equation}
to distinguishing the individual parts. Fourth, we need the multiplicity of each part of the partition $\pi_{ij}$, denoted $\pi_{ij}^{\star}$, so with the partition $1+1+3+3 = 8$ represented as a sequence is $\pi_{ij}=(1,1,3,3)$, and then $\pi_{ij}^{\star}=(0,2,0,2,0)$, so two 1's and two 3's. Fifth, we need the notion of the complement of a Ferrers diagram. For $1\leq i < j \leq k-1$ we let $\pi_{j,i}$ denote
the complement of $\pi_{i,j}$. For example, the diagram of $\pi_{i,j}=(1,1,3,3)$ has complement $\pi_{j,i}=(0,2,2,4)$, and $\pi_{j,i}^{\star}=(1,0,2,0,1)$. Finally, for some lattice path $\Pi_{\gamma}$, another lattice path $\Pi$, and some integer partition $\pi(\Pi)$, we define the ``dot product'' of two equal length sequences
\begin{equation}
  S(\Pi_{\gamma}) \cdot \pi(\Pi) :=  S_0 (\Pi_{\gamma}) \pi (0) + \cdots +  S_{m_{\gamma}} (\Pi_{\gamma}) \pi (m_{\gamma}).
    \end{equation}
A simple example of the use of all the notation in this article is given in detail in Section \ref{sec:example}.

\subsection{Results}
For density $\lambda \geq 0$, number of hops $k \in \mathbb{N}^{+}$, a connection range $r_0$ satisfying
\begin{equation}\label{e:range1}
\frac{1}{k} < r_0 < \frac{1}{k-1}
\end{equation}
and with $|x-y|$ the Euclidean distance between two points $x,y \in [0,1]$, consider the 1d hard random geometric graph on the Poisson point process $\mathcal{P}_{\lambda} \subset [0,1]$, and two further points conditioned to exist at $0,1 \in [0,1]$, forming the vertex set $V := P_{\lambda} \cup \{0,1\}$, and with edge set $E := \{\{x,y\} \in V^{2}: |x-y|<r_0 \}$, denoted $G(V,E)$. Consider $m_1,\dots,m_{k-1} \sim \textrm{Poisson}(\lambda|L|)$, where $|L|$ is given by Eq.~\eqref{e:width}.  Then we have the following statements for the distribution of the number of $k$-hop paths $\sigma_{k}$ between the endpoints $0,1$.

Firstly concerning the trivial case when $k=1$, then $\sigma_{1}=\mathbf{1}_{\{r_{0}\geq1\}}$. When $k=2$, the Poisson number $m_1$ of vertices in $L_1$ implies $\sigma_{2} \sim \text{Poisson}(2r_0 - 1)$. Now consider the case $k \geq 3$ of the following theorem.
\begin{theorem}[Distribution of $\sigma_{k}$]\label{t:p1-2}
Assume that $k \geq 3$. The probability of observing $n$ $k$-hop paths connecting the two vertices at the boundary of $[0,1]$ in $G(V,E)$ is given by
 \begin{equation}\label{e:eq1}
  \mathbb{P} (\sigma_k =n ) = \frac{ 1}{{m_1 + \dots +  m_{k-1} \choose m_1,\ldots, m_{k-1} }}  \sum_{\Pi_{k-2} \in \mathcal{P}\left(\mathcal{M}_{k-2}\right)} \sum_{\Pi \in {\cal P}([m_{k-2}\times m_{k-1}])   \atop   S(\Pi_{k-2}) \cdot \pi^\star(\Pi)   = n }
  \prod_{r=0}^{m_{k-2}} {\pi^\star (r)+\sum_{l=1}^{k-3} \pi_{k-2,l}^\star (r) \choose \pi^\star (r)}, 
 \end{equation}
\end{theorem}

 Using the coefficient extraction operator $[u^{n}]$ of Eq. \eqref{e:ceo1}, which is standard notation for the coefficient of $u^{n}$ in a series it precedes, we have the following result.
\begin{theorem}[Probability generating function for $\sigma_{k}$] \label{t:p1-4}
  Let $k \geq 2$. Then the p.g.f. of the number of $k$-hop paths connecting  the two vertices at the boundary of $[0,1]$ in $G(V,E)$ is given by
\begin{equation}
  \label{e:final}
  \mathbb{E}\big[q^{\sigma_k}\big] = \frac{ 1}{{m_1 + \cdots +  m_{k-1} \choose m_1,\ldots, m_{k-1} }}
         [u^{m_{k-1}}] \sum_{\Pi \in \mathcal{P}\left(\mathcal{M}_{k-2}\right)}
         \prod_{t=0}^{m_{k-2}}\frac{1}{( 
      1 - u q^{S_{t}(\Pi)} )^{1+\sum_{l=1}^{k-3}\pi_{k-2,l}^\star (t)}} .
\end{equation}
\end{theorem}
The coefficients of Eq. \ref{e:final} are corroborated via Monte Carlo simulation of the 1d RGG for the case $k=3,4$ in Figs. \ref{fig:k81} and \ref{fig:k82}.

\subsection{The three-hop case}

\begin{corollary}[The distribution in the case $k=3$]
  Theorem \ref{t:p1-2} implies that, since $\mathcal{P}(\mathcal{M}_{1})$ contains only one path, the probability of observing $n$ three-hop paths connecting the two vertices at the boundary of $[0,1]$ in $G(V,E)$ is given by
 \begin{equation}\label{e:3hopsprop}
  \mathbb{P} (\sigma_3 =n ) = \frac{ 1}{{m_1 + m_2\choose m_1}} \sum_{\Pi \in {\cal P}(\mathcal{M}_{2})\atop S(\Pi_{2}) \cdot \pi^\star(\Pi)  = n }1, 
 \end{equation}

\end{corollary}
\begin{corollary}[The p.g.f. in the case $k=3$]
  Theorem \ref{t:p1-4} implies that the p.g.f. of the number of $3$-hop paths connecting  the two vertices at the boundary of $[0,1]$ in $G(V,E)$ is given by a normalised $q$-binomial coefficient of Eq.~\eqref{e:qbin},
\begin{equation}\label{e:threepgf}
 \mathbb{E}\big[q^{\sigma_3}\big] =  \frac{ 1}{{m_1 + m_2\choose m_1}}  [u^{m_{2}}] \prod_{t=0}^{m_{1}}\frac{1}{ 1 - u q^{t}  } = \frac{ 1}{{m_1 + m_2\choose m_1}} {m_1 + m_2 \choose m_1}_{q}.
\end{equation}
\end{corollary}

\section{Preliminaries}\label{sec:preliminaries}

%\begin{figure}
%  \centering \includegraphics[scale=0.9]{pic2}
%  \caption{Two-hop paths (red) between the endpoint vertices (vertex 1 at coordinates (0,0) and vertex 8 at coordinate (1,0)) in a 1d random geometric graph on 8 Poisson points. The lens $L_1$ on the interval $[0,1]$ is shown as the black rectangluar box, where nodes in red fall on two-hop paths. The vertices of the Poisson point process $\P_{\lambda}$, not on paths, are shown in black.}
%  \includegraphics[scale=0.9]{pic3}  \caption{Three-hop paths (red) between the endpoint vertices (vertex 1 at coordinates (0,0) and vertex 18 at coordinate (1,0)) in a 1d random geometric graph on 8 Poisson points. The lenses $L_1$ and $L_2$, on the interval $[0,1]$, are shown as the two black rectangluar boxes, where nodes in red fall on the three-hop paths. The vertices of the Poisson point process $\P_{\lambda}$, not on paths, are shown in black.}
  
 % \includegraphics[scale=0.9]{pic4} \caption{Four-hop paths (red) between the endpoint vertices (vertex 1 at coordinates (0,0) and vertex 12 at coordinate (1,0)) in a 1d random geometric graph on 8 Poisson points. The lenses $L_1,L_2$ and $L_3$, on the interval $[0,1]$, are shown as the three black rectangluar boxes, where nodes in red fall on the three-hop paths. The vertices of the Poisson point process $\P_{\lambda}$, not on paths, are shown in black.}
  
 % \includegraphics[scale=0.9]{pic5}
 % \caption{Five hop paths.}
% \par\end{centering}
 %   \label{fig:main-222}
%\end{figure}
 The following preliminary sections introduce our notation, the 1d RGG, multidimensional lattice paths and their projections, and the associated restricted integer partitions.
 \subsection{1d random geometric graphs}\label{sec:1drgg}
%\begin{figure}[t!]
 %              \begin{adjustbox}{width=\linewidth}
 %    \centering
%\includegraphics[scale=0.4]{r3new2}
 %           \end{adjustbox}
%                \caption{Three integer partitions represented as lattice paths, each enclosing an area of 12 units. The number of ways this can be done is given by the coefficient of $q^{12}$ in the series expansion of the Gaussian binomial coefficient ${8 \choose 4}_{q}$, which is 5.}
% \par\end{centering}
%    \label{fig:main2}
%\end{figure}
 To define the homogeneous Poisson point process $\mathcal{P}_{\lambda} \subset [0,1]$, of intensity $\lambda$ times Lebesgue measure $||\cdot||$ on $[0,1]$, for $a,b \in [0,1]$, with $a \leq b$, and with $N(a,b]$ the count of points in $\mathcal{P}_{\lambda} \cap (a,b]$, then we have that
   \begin{equation}
     \mathbb{P}(N(a,b] = n) = \frac{\left(\lambda\left(b-a\right)\right)^{n}}{n!}e^{-\lambda(b-a)},
   \end{equation}
   and furthermore, that the counts of points in any pair of disjoint intervals of $[0,1]$ are independent.
   Then, for density $\lambda \geq 0$ points per unit length, connection range $r_0>0$, and with $|x-y|$ the Euclidean distance between two points $x,y \in [0,1]$, then with the Poisson point process $\mathcal{P}_{\lambda} \subset [0,1]$, and two further points conditioned to exist at $0,1 \in [0,1]$, forming the vertex set $V := P_{\lambda} \cup \{0,1\}$, and with edge set $E := \{\{x,y\} \in V^{2}: |x-y|<r_0 \}$, then the 1d hard random geometric graph is the graph $G(V,E)$. This model is used throughout this article.
\begin{figure}[t]
\centering \includegraphics[scale=0.38]{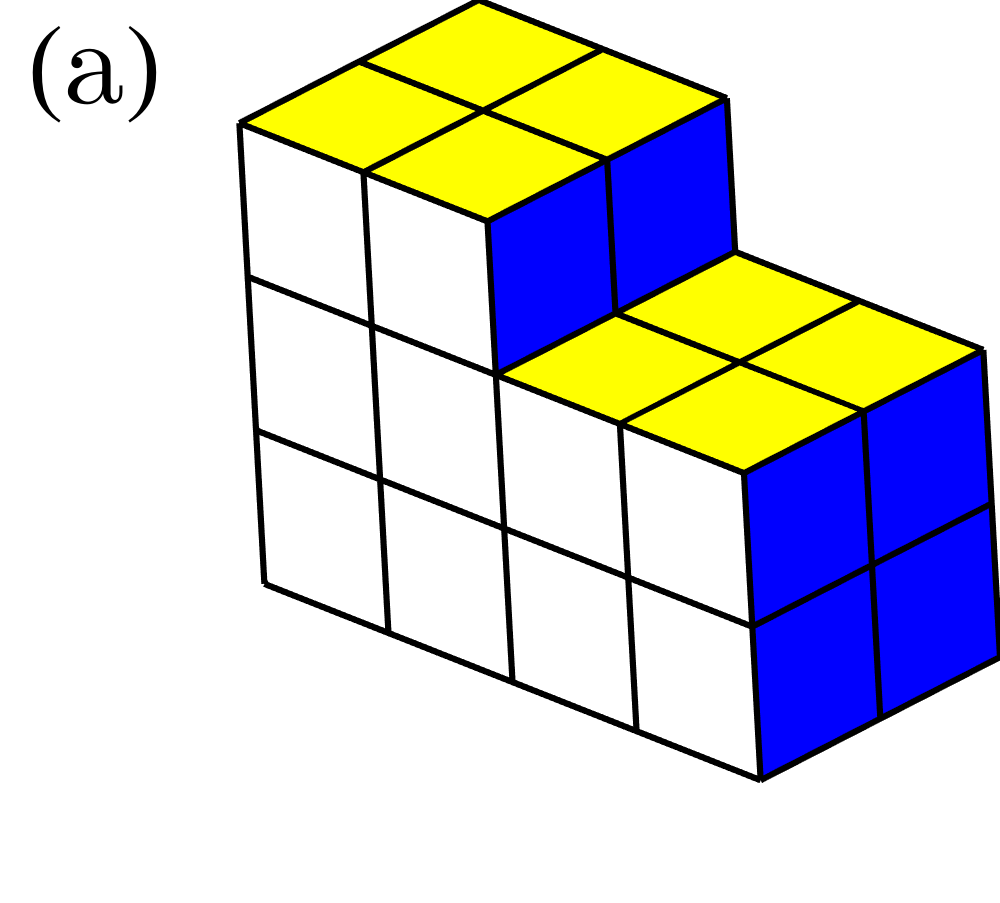} \hspace{2mm} \includegraphics[scale=0.38]{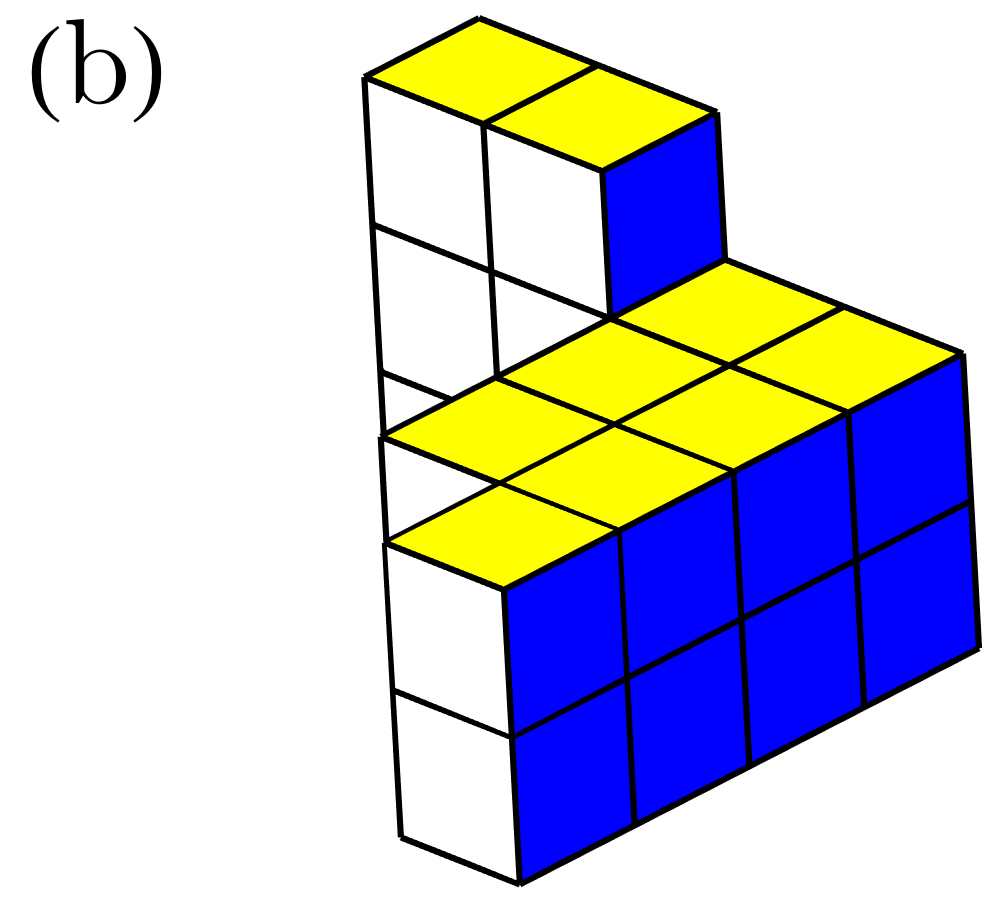} \hspace{2mm} \includegraphics[scale=0.38]{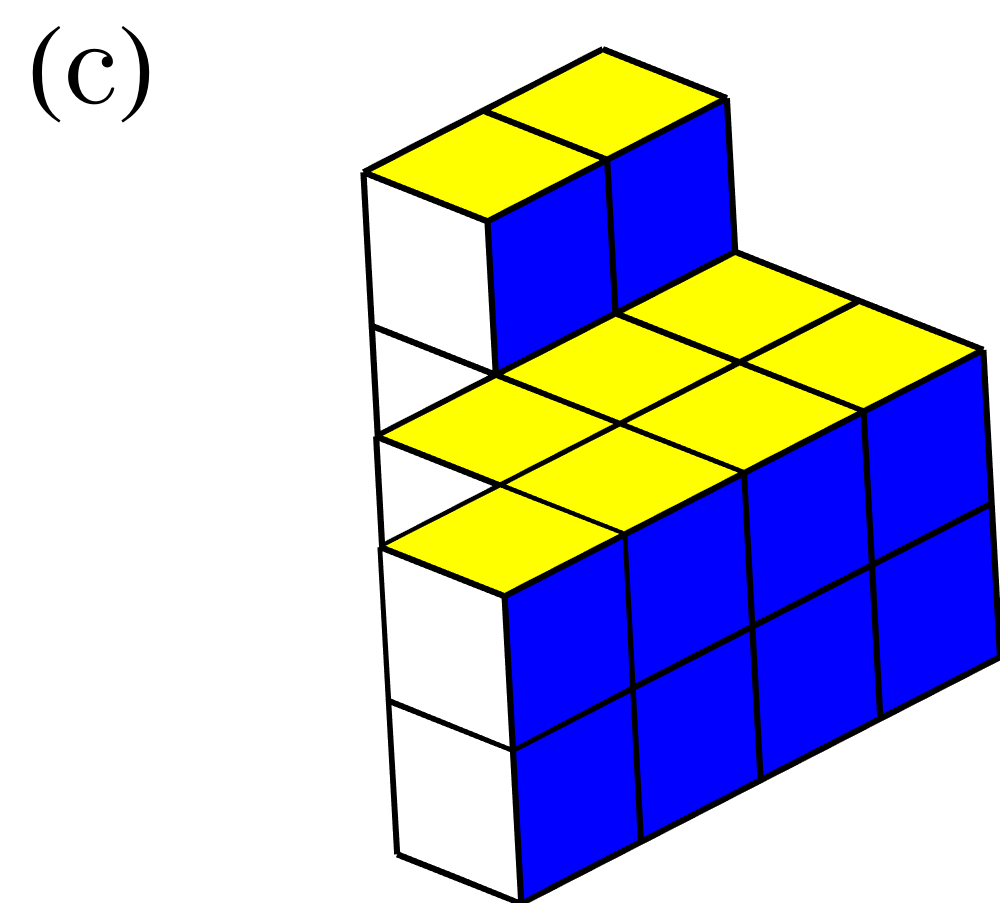} \caption{The three restricted integer partitions of 20 using exactly $m_3=4$ parts from the restricted set of integers $S(\Pi_{3})=\{0,2,4,7,10\}$. (a) $0+0+10+10 \vdash 20$, (b) $2+4+4+10 \vdash 20$ and (c) $2+4+7+7 \vdash 20$.}\label{fig:main-222}
\end{figure}

\subsection{Lattice paths}
%\begin{figure}[H] % 
%\centering
%\hskip-0.1cm
%\begin{subfigure}{.33\textwidth}
%\centering
%\includegraphics[scale=0.35]{h0}
%\caption{\small First path.} 
%\label{p-a} 
%\end{subfigure}
% \hskip-0.1cm
%\begin{subfigure}{.33\textwidth}
%\centering
%\includegraphics[scale=0.35]{h1}
%\caption{\small Second path.} 
%\label{p-b} 
%\end{subfigure}
%\begin{subfigure}{.33\textwidth}
%\centering
%\includegraphics[scale=0.35]{h2}
%\caption{\small Third path.} 
%\label{p-c} 
%\end{subfigure}
%\caption{The 3d volume (purple) under three different lattice paths (a), (b) and (c), in black. Each face of the cubes under the path is under the respective projection of the path onto each of the six faces of the domain. This is depicted in Figure~\ref{fig:project}.}
% \par\end{centering}
  %  \label{fig:main-1}
%\end{figure}
In layman's terms, a lattice path is a path from the lower left to top right of a rectangular lattice. It may be high-dimensional, so between the extreme corners of a $d$-dimensional lattice. The area under a lattice path is major topic in combinatorics.
\begin{definition}\label{def:paths}
  % Given $n\geq 1$,
  Given two lattice points $A$ and $\Omega$, a set of steps $\mathbb{S}$, and an integer $m>0$, we denote by $L_{m}(A \to \Omega ;\mathbb{S})$ the set of lattice paths from $A$ to $\Omega$ with $m$ steps in the set $\mathbb{S}$.
\end{definition}
\begin{definition}
With $e_i$ denoting the vector with a 1 in the $i^{\text{th}}$ position, and 0 elsewhere, we consider
\begin{equation}
\mathcal{P}([m_{k_{1}} \times \dots \times m_{k_{n}}]) := L_{m}(\mathbf{0} \to  (m_{k_{1}},\dots,m_{k_{n}}) ; \{e_i: 1 \leq i \leq n \} )
\end{equation}
to be the set of lattice paths between $\mathbf{0}$ and $(m_{k_{1}},\dots,m_{k_{n}})$ in the lattice $[m_{k_{1}} \times \dots \times m_{k_{n}}]$.
\end{definition}Consequently,
\begin{equation}
\mathcal{P}(\mathcal{M}_{d}) = L_{m}(\mathbf{0} \to  \mathbf{m} ; \{e_i: 1 \leq i \leq d \} )
\end{equation}
is the set of lattice paths between $\mathbf{0}$ and $\mathbf{m}$ in $ \mathcal{M}_{d}$.

The lattice paths of Def. \ref{def:paths} resemble natural $d$-dimensional versions of the up-right lattice paths in $\mathcal{M}_{2}$. We can equivalently regard $\mathcal{P}({\cal M}_{d})$ as the set of
permutations of the multiset
\begin{equation}\label{e:multiset1}
  e_{1}^{m_1}e_{2}^{m_2}\dots e_{d}^{m_d} :=  \{\underbrace{e_{1},e_{1},\dots,e_{1}}_{m_1\text{ times}},\underbrace{e_{2},e_{2},\dots,e_{2}}_{m_2\text{ times}},\dots,\underbrace{e_{d},e_{d},\dots,e_{d}}_{m_d \text{ times}}\}
\end{equation}
where direction $e_{1}$ occurs $m_1$ times, $e_{2}$ occurs $m_2$ times, and so on, until all $m_1+\dots+m_{d}$ directed lattice steps have been taken, and the path arrives at the boundary point $\mathbf{m}$.

\begin{figure}[t]
\centering \includegraphics[scale=1.2]{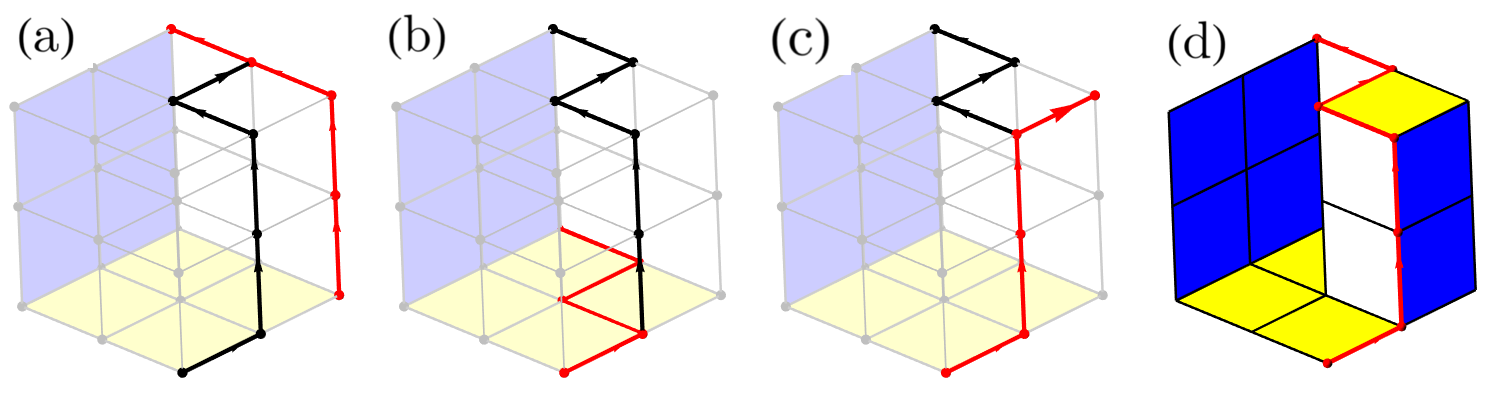}  \caption{Projecting the path onto the faces of the lattice, determining the volume beneath. With $\Pi_3=\left(e_3,e_1,e_1,e_2,e_3,e_2\right)$ (black lattice path), the steps (a)-(c) show the three projections $\Pi_{12}(\Pi_{3})=\{e_1,e_1,e_2,e_2\},\Pi_{13}(\Pi_{3})=\{ e_3,e_2,e_3,e_2\}$ and $\Pi_{23}(\Pi_{3})=\{e_3,e_1,e_1,e_3\}$ (red lattice paths) onto the white, yellow and blue faces respectively, of the lattice $\mathcal{M}_{3}=[2 \times 2 \times 2]$, used to determine the exact 3d volume shown in (d) (stacked cubes).}\label{fig:project}
\end{figure}

\subsubsection{Projections}\label{sec:projections}
In addition, for $1\leq i < j \leq d$ we denote by
$\Pi_{i,j}(\Pi_{k-1})$ the projection of the path $\Pi_{k-1}$ onto the $2$-dimensional face
$m_i\times m_j$ generated by $(e_i,e_j)$. Therefore
\begin{equation}\label{e:lproj}
  \Pi_{i,j}(\Pi_{k-1}) \in L_{m_{i}+m_{j}}(\mathbf{0} \to m_i e_i + m_j e_j ; \{e_i, e_j\}).
 \end{equation}
This is depicted in Fig.~\ref{fig:project}, where we have $e_1=(1,0,0)$ as \textit{right}, $e_1=(0,1,0)$ as \textit{up}, $e_1=(0,0,1)$ as \textit{in}, $\Pi_{2}=(e_1,e_2,e_2,e_3,e_1,e_3)$, and so $\Pi_{1,2}(\Pi_{2}) = (e_1,e_2,e_2,e_1)$, $\Pi_{2,3}(\Pi_{2})= (e_2,e_2,e_3,e_3)$, and $\Pi_{1,3}(\Pi_{2}) = (e_1,e_3,e_1,e_3)$. In the specific example of Fig.~\ref{fig:project}, since $\pi_{1,2}=(2,2)$,  $\pi_{2,3}=(0,1)$ and $\pi_{1,3}=(0,2)$, we have the restricted set of integers $S(\Pi_{3})=(0,0+2,0+2+2)=(0,2,4)$, and the partition $\pi_{2,3}^{\star} \cdot S(\Pi_{3}) \vdash 2 = (1,1,0) \cdot (0,2,4) \vdash 2= 0+2 \vdash 2$. This black lattice path is the only lattice path which corresponds to this restricted integer partition of $0+2 \vdash 2$, and so in this case $\text{Degeneracy}\left(\pi_{2,3}^{\star} \cdot S(\Pi_{3}) \vdash 2\right) = \text{Degeneracy}\left(\left(1,1,0 \right) \cdot \left(0,2,4 \right) \vdash 2)\right)=1$.

 \subsubsection{Volume under the lattice path}\label{s:ip}
 In Section \ref{sec:projections}, we discussed the projections $\Pi_{ij}(\Pi_{k-1})$ of $\Pi_{k-1}$ onto the faces of $\mathcal{M}_{k-1}$. These projections are themselves Ferrers diagrams, depicted in Fig.~\ref{fig:project}. Therefore, given a lattice path $\Pi_{k-1}$, consider the integer partition $\pi_{ij}\left(\Pi_{k-1}\right)$, which is a sequence of $m_i$ integers each of size no larger than $m_j$, and whose $t^{\text{th}}$ part is $\pi_{ij}(t)$. The context will be clear that we are considering by $\pi_{ij}(t)$ the $t^{\text{th}}$ part of the integer sequence $\pi_{ij}\left(\Pi_{k-1}\right)$.

We now define the volume under the lattice path. Consider the $(k-1)$-dimensional lattice cell
 \begin{align}
  \mathrm{cell}\left(x_1,\dots,x_{k-1}\right) := [x_1,x_1+1] \times [x_2,x_2+1] \times \dots \times [x_{k-1},x_{k-1}+1]
 \end{align}
 and its $(u,v)$-face
  \begin{align}
  \mathrm{cell}_{u,v}\left(x_1,\dots,x_{k-1}\right) := [x_u,x_u+1] \times [x_v,x_v+1]
 \end{align}
Then, we mean that the $(k-1)$-dimensional volume under the lattice path $\Pi_{k-1}$ is
 \begin{align}\label{e:v}
  V(\Pi_{k-1}) := \# \{\mathrm{cell}_{u,v}\left(x_1,\dots,x_{k-1}\right) : x_{v} < \pi_{u,v}(x_{u}), 1 \leq u < v \leq k-1\}.
 \end{align}

 \subsubsection{Volume under the path as a restricted integer partition}
The volume $V$ defined in Eq. \ref{e:v} is an integer partition in the following way. Consider the lattice path $\Pi_{k-1}$ enclosing a volume $ V(\Pi_{k-1})$. Then consider the integer sequences
 \begin{align}
  S(\Pi_{2}) &:= \left( 0,1,\dots,m_1 \right) \\
  S(\Pi_{3}) &:= \Bigg(\sum_{i_{1}=1}^{t}\pi_{1,2}(i_1)\Bigg)_{t=0,\ldots,m_{2}}\\
  S(\Pi_{4}) &:= \Bigg(\sum_{i_{2}=1}^{t} \sum_{i_{1}=1}^{\pi_{2,3}(i_{2})}\pi_{1,2}(i_1)\Bigg)_{t=0,\ldots,m_{3}}\\
  S(\Pi_{5}) &:=\Bigg(\sum_{i_{3}=1}^{t} \sum_{i_{2}=1}^{\pi_{3,4}(i_{3})}\sum_{i_{1}=1}^{\pi_{2,3}(i_{2})}\pi_{1,2}(i_1)\Bigg)_{t=0,\ldots,m_{4}}\\
 &\;\;\; \vdots \nonumber \\
S(\Pi_{k-1}) &:= 
\Bigg(
\sum_{i_{k-3}=1}^{t} \sum_{i_{k-2}=1}^{\pi_{k-3,k-2}(i_{k-3})} \cdots  \sum_{i_{1}=1}^{\pi_{2,3}(i_{2})}
\pi_{1,2}(i_{1})
\Bigg)_{t = 0,\ldots , m_{k-2}}.\label{e:ss1}
\end{align}
where the empty sum is the first element in each sequence $S$, since the case $t=0$ gives a sum with no summands each time, and so gives zero by definition.
 We write $S(\Pi_{k-2}) = \left(S_{i}(\Pi_{k-2})\right)_{i=0,\dots,m_{k-3}}$ to distingusihing the individual parts. Recall that the multiplicity of each part of the partition $\pi_{ij}$ is denoted $\pi_{ij}^{\star}$, so for example the partition $8=1+1+3+3$ may be represented by the sequence $\pi_{ij}=(1,1,3,3)$, and $\pi_{ij}^{\star}=(0,2,0,2,0)$, where we have considered the case where we use four parts from the set $\{0,1,2,3,4\}$.

 Now consider, alongside Eq.~\eqref{e:ss1}, the integer partition $\pi_{k-2,k-1}(\Pi_{k-1})$. Then, consider the integer $n$, and the integer partition $\Lambda(\Pi_{k-1}) \vdash n$ corresponding to the volume under the lattice path. Then 
 \begin{equation}
 \Lambda(\Pi_{k-1}) = \pi_{k-2,k-1}^{\star} \cdot  S(\Pi_{k-2}).
    \end{equation}
$\Lambda(\Pi_{k-1})$ is therefore a restricted integer partition of the volume under the path $\Pi_{k-1}$, using exactly $m_{k-1}$ parts from the set $S(\Pi_{k-2})$. This is depicted in Fig. \ref{fig:main-222}, where the three restricted integer partitions of 20 using exactly $m_3=4$ parts from the restricted set of integers $S(\Pi_{3})=\{0,2,4,7,10\}$ are shown.

 \subsection{Integer partitions and generating functions}
We use the notation $\Lambda \vdash n$ to denote that $\Lambda$ is an integer partition of $n$. For example, with $n=20$, we write $2+4+4+10 \vdash 20$ when $\Lambda=(2,4,4,10)$. The \textit{partition function} $p(n)$ counts the number of possible partitions of an integer $n$. Two sums differing only by the order of their summands are not considered to be distinct partitions. They are however distinct \textit{compositions}. It is common to use a generating function
\begin{equation}
\sum_{n=0}^\infty p(n)q^n = \prod_{r=1}^\infty \frac {1}{1-q^r}. 
% =(1+q+q^2+q^3+\cdots)(1+q^2+q^4+q^6+\cdots)\cdots, 
\end{equation} 
% For a ring $R$ with unity we denote by $R[q,u]$ the ring of polynomials in $q,u$ with coefficients in $R$, and we let $R[[q,u]]$ denote the ring of formal power series of the form \begin{equation}  \sum_{k,n \geq 0}a_{n,k} q^{n} u^{n} \end{equation} with coefficients $a_{n,k} \in R$, $n,k\geq 0$, which are formal sums of monomials $q^{i}u^{j}$.  
Consider the number $p(n,r)$ of integer partitions of $n$ into exactly $r$ parts. Then we have a generating function in two variables,
\begin{equation}
  \sum_{r=0}^\infty u^{r}
  \sum_{n=r}^\infty p(n,r) q^n = \prod_{r=1}^\infty \frac {1}{1-u q^r}, 
\end{equation}
and define the linear \textit{coefficient extraction operator} 
 $[u^{n}]A(u)$, acting on any formal power series 
\begin{equation}
  A(u) =  \sum_{n \geq 0}a_{n} u^{n}
  \end{equation}
  as the operator which extracts the $n^{\text{th}}$ coefficient in the series,
 \begin{equation}\label{e:ceo1}
  [u^{n}]A(u) := a_n, \qquad n \geq 0.
 \end{equation}
Then, we may write the integer partition function
\begin{equation}
   p(n,r) = [q^{n}][u^{r}]\prod_{r=1}^\infty \frac {1}{1-u q^r}, 
\end{equation}
Given a set of integers $S$, the number of integer partitions $p_S(n,r)$ of $n$ into $r$ parts chosen from the set $S$ satisfies
 \begin{equation}\label{e:fmain} 
  \sum_{r=0}^\infty u^{r}
  \sum_{n=r}^{r\max ( S) } p_S (n,r)q^n
 = \prod_{k \in S} \frac {1}{1-u q^k}, 
\end{equation} 
 The probability generating function of the area under a uniformly random lattice path from $(0,0)$ to $(m_1,m_2)$ in $\mathcal{M}_{2}$ is given by the $q$-binomial coefficient, which is, for $r \leq m$,
 \begin{equation}\label{e:qbin}
{m \choose r}_q=
\frac{(1-q^m)(1-q^{m-1})\cdots(1-q^{m-r+1})} {(1-q)(1-q^2)\cdots(1-q^r)}
 \end{equation}
 This may be written
\begin{equation}
{m_1+m_2 \choose m_1}_q
= \sum_{n=0}^{ m_1 m_2 } p_{\{ 0,1, \ldots , m_1 \}} (n,m_2)q^n.
\end{equation}
giving a generating function for restricted integer partitions in $m_2$ parts no larger than $m_1$.

\subsection{Example case}\label{sec:example}
Consider Figs. \ref{fig:3hops}, \ref{fig:3hops2}, \ref{fig:used} and \ref{fig:k4-1}. A simple example case is the following, for the case $k=4$. We have the cuboid lattice $\mathcal{M}_{3} = [m_1 \times m_2 \times m_3]$. A lattice path starting at $(0,0,0)$ and terminating at $(m_1,m_2,m_3)$ consists of three types of movements, $e_1$, $e_2$ and $e_3$ (right, up, and in). The lattice path $\Pi_{k-1}=\Pi_{3}$ encloses a 3d volume beneath, bounded by the the domain walls, and itself, in a new way discussed in Section \ref{s:ip}. Each of the six projections $\Pi_{1,2},\Pi_{1,3},\dots,\Pi_{3,2}$ onto the six faces of the cuboid are themselves Ferrers diagrams. Consider the example lattice path shown in Fig. \ref{fig:main-22} (a), which is $\Pi_{4} = (\text{up},\text{up},\text{in},\text{right},\text{in},\text{right},\text{right},\text{up},\text{in},\text{in},\text{right},\text{up})$. Then we have, on the white face, $\Pi_{1,2}=(\text{up},\text{up},\text{in},\text{in},\text{up},\text{in},\text{in},\text{up})$ and therefore the corresponding partition $\pi_{1,2}=(2,2,3,3)$, as well as, on the blue face, \begin{equation}\Pi_{1,3}=(\text{up},\text{up},\text{right},\text{right},\text{right},\text{up},\text{right},\text{up})\end{equation} and therefore the corresponding partition $\pi_{1,3}=(2,2,2,3)$, and finally on the yellow face $\Pi_{2,3}=(\text{in},\text{right},\text{in},\text{right},\text{right},\text{in},\text{in},\text{right})$ and therefore the corresponding partition $\pi_{2,3}=(1,2,2,4)$. The projections $\Pi_{2,1},\Pi_{3,1}$ and $\Pi_{3,2}$ are the same as their counterparts $\Pi_{1,2},\Pi_{1,3}$ and $\Pi_{2,3}$, since they represent the projection of the lattice path onto the opposing face. In our case, we therefore have
\begin{align}
  S(\Pi_{3}) &= \Bigg(\sum_{i_{1}=1}^{t} \pi_{1,2}(i_1)\Bigg)_{t=0,1,\ldots,4}\\
  &= \left(0,\pi_{1,2}(1),\pi_{1,2}(1)+\pi_{1,2}(2),\dots,\pi_{1,2}(1)+\pi_{1,2}(2)+\pi_{1,2}(3)+\pi_{1,2}(4)\right)\\
    &= \left(0,2,4,7,10\right)
\end{align}
For the multiplicities and part counts, we have $\pi_{2,1}=(0,0,2,4)$, $\pi_{3,1}=(0,0,3,4)$, $\pi_{3,2}=(0,1,3,4)$, $\pi_{2,1}^{\star}=(2,0,1,0,1)$, $\pi_{3,1}^{\star}=(2,0,0,1,1)$ and $\pi_{3,2}^{\star}=(1,1,0,1,1)$. We would also have
\begin{align}
  S(\Pi_{4}) &= \Bigg(\sum_{i_{2}=1}^{t} \sum_{i_{1}=1}^{\pi_{2,3}(i_{2})}\pi_{1,2}(i_1)\Bigg)_{t=0,\ldots,4}\\
  &= \Bigg(0,\sum_{i_{1}=1}^{1}\pi_{1,2}(i_1),\dots, \sum_{i_{1}=1}^{1}\pi_{1,2}(i_1)+\sum_{i_{1}=1}^{2}\pi_{1,2}(i_1)+\sum_{i_{1}=1}^{2}\pi_{1,2}(i_1)+\sum_{i_{1}=1}^{4}\pi_{1,2}(i_1)\Bigg)\\
    &= \left(0,2,\dots,2+(2+2)+(2+2+3)+(2+2+3+3)\right)\\
    &= \left(0,2,6,10,20\right)
\end{align}

\section{Proofs}\label{sec:allres}

\begin{figure}[t]
\centering \includegraphics[scale=0.38]{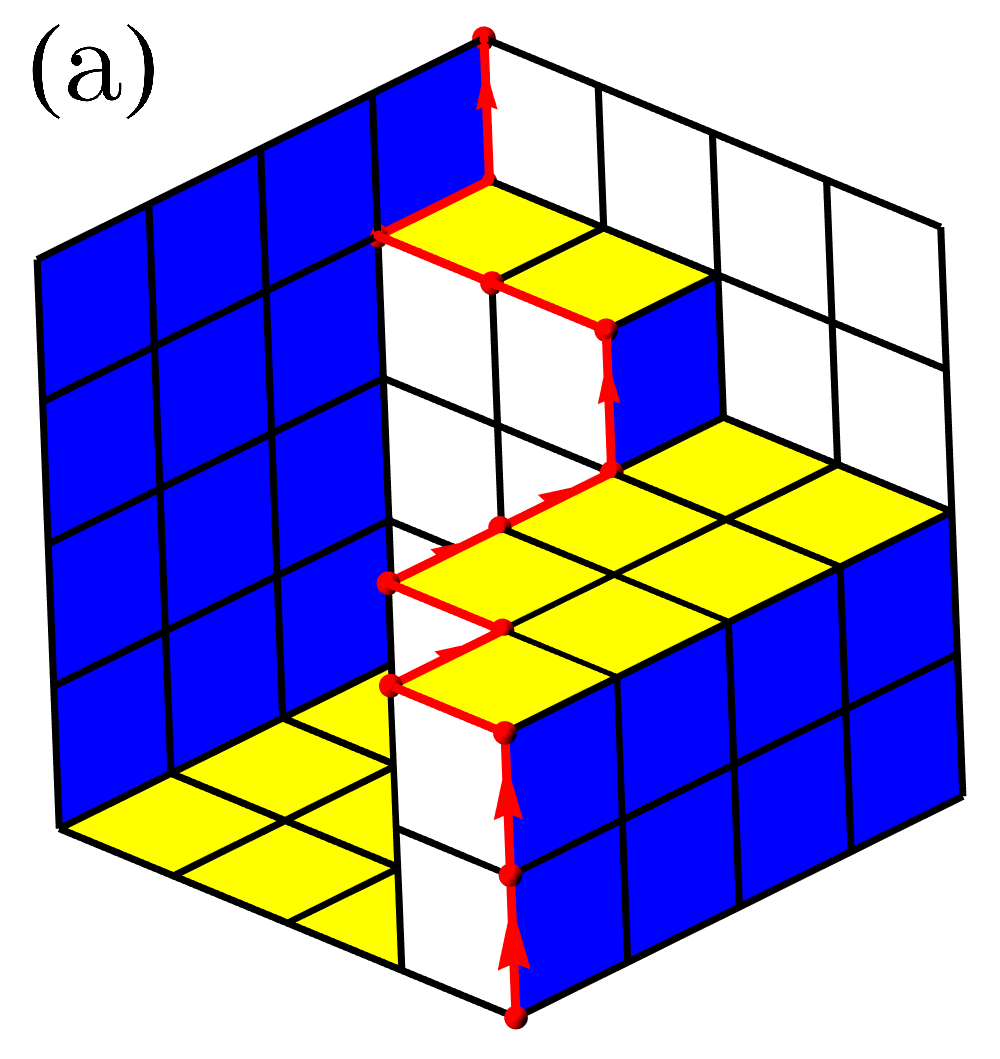} \hspace{2mm} \includegraphics[scale=0.38]{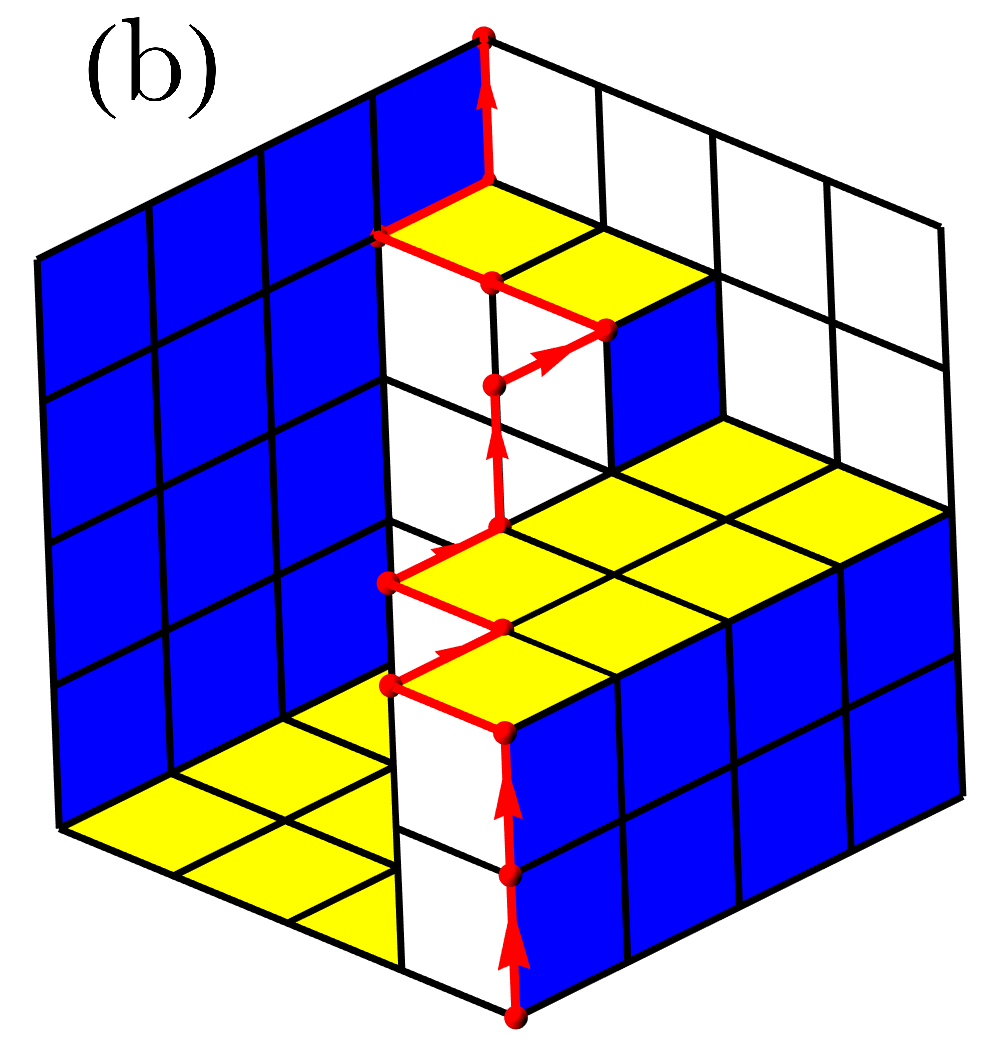} \hspace{2mm} \includegraphics[scale=0.38]{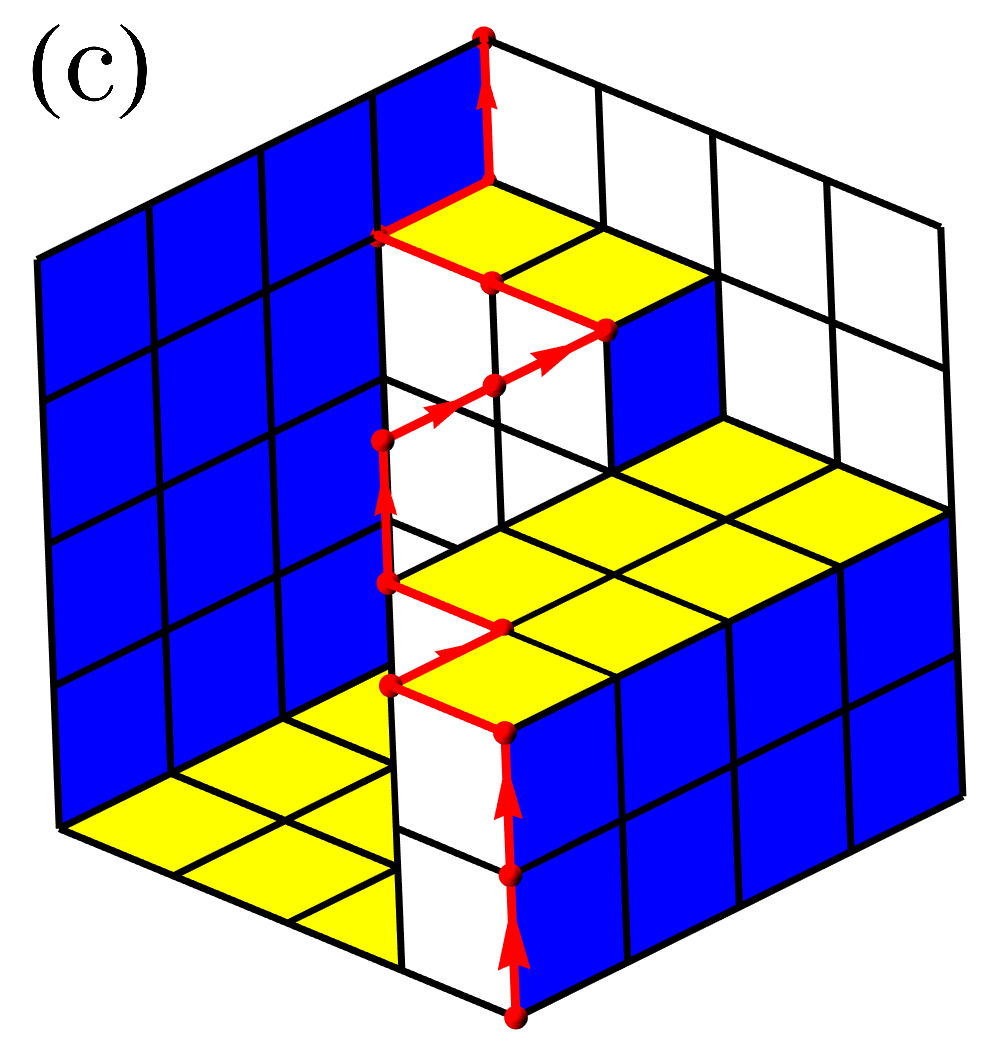} \\ \includegraphics[scale=0.38]{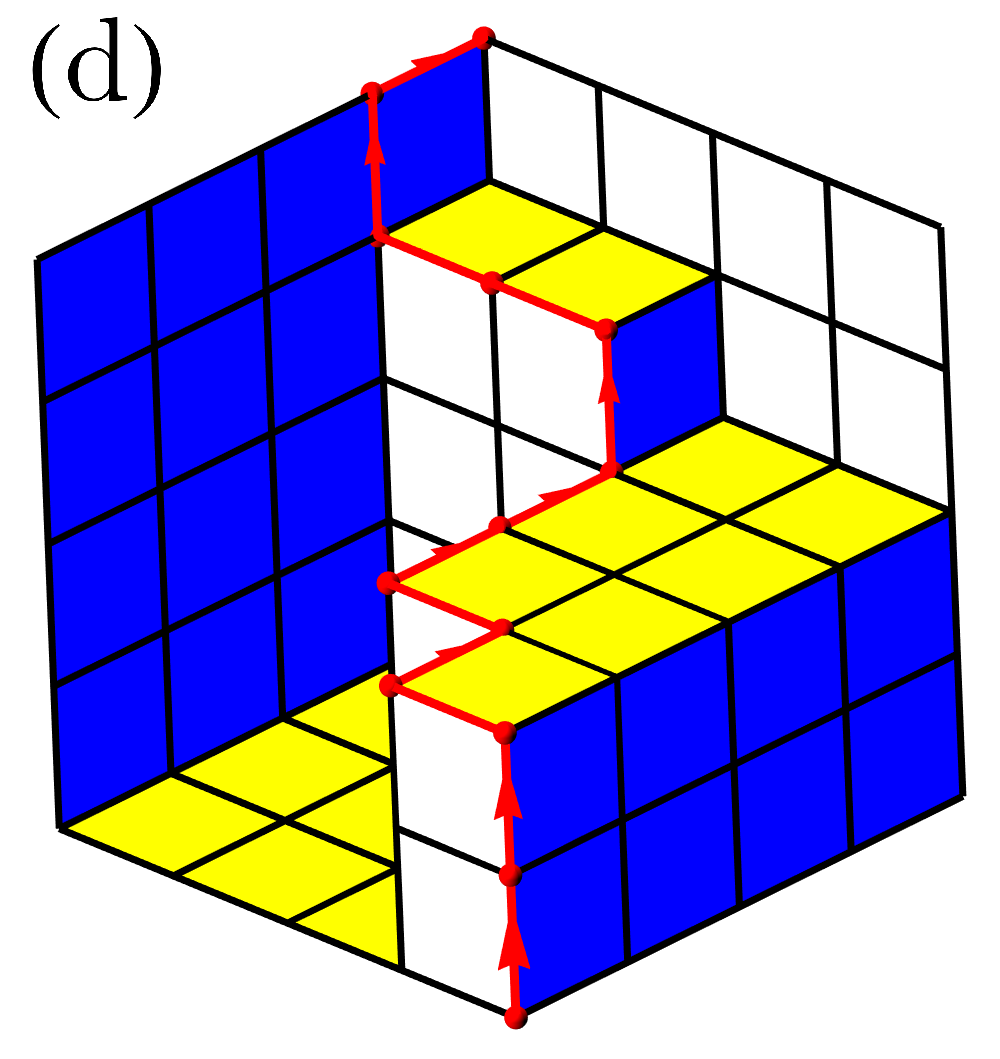} \hspace{2mm} \includegraphics[scale=0.38]{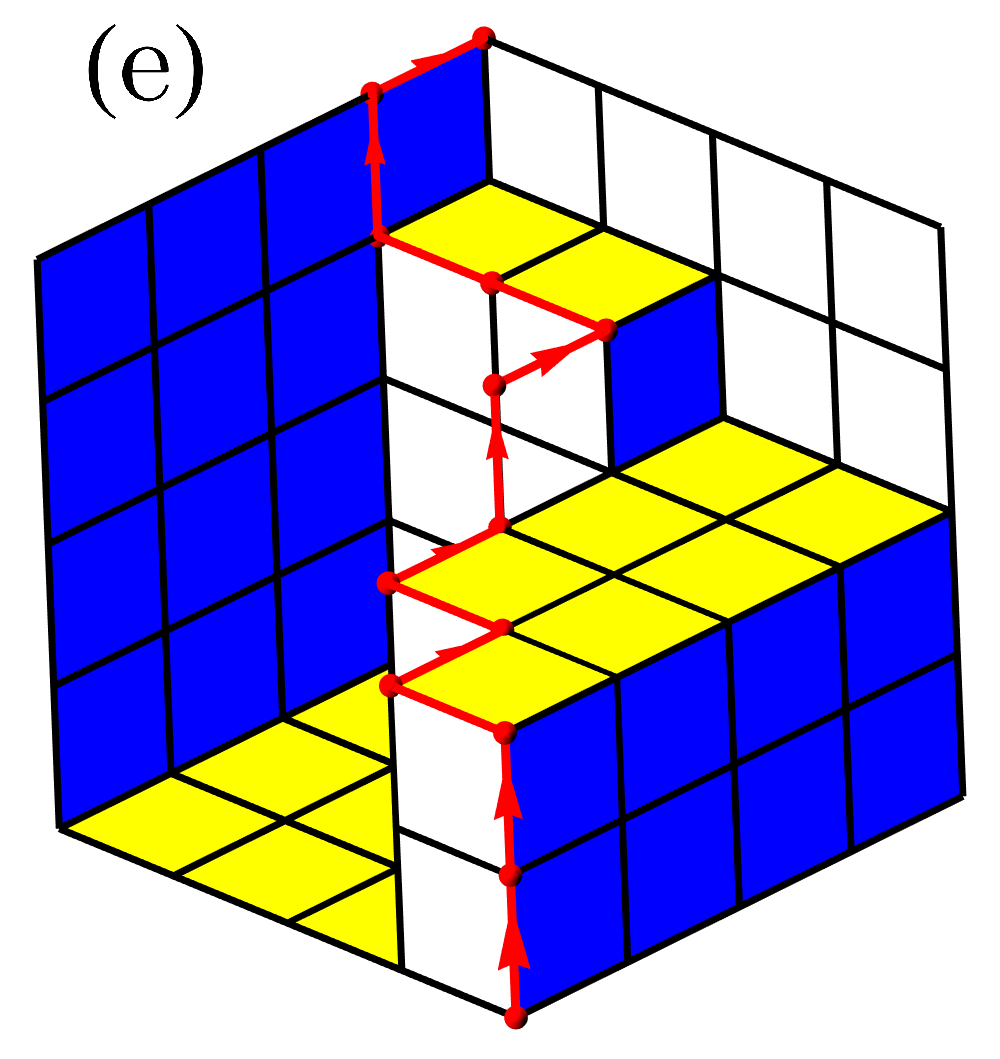} \hspace{2mm} \includegraphics[scale=0.38]{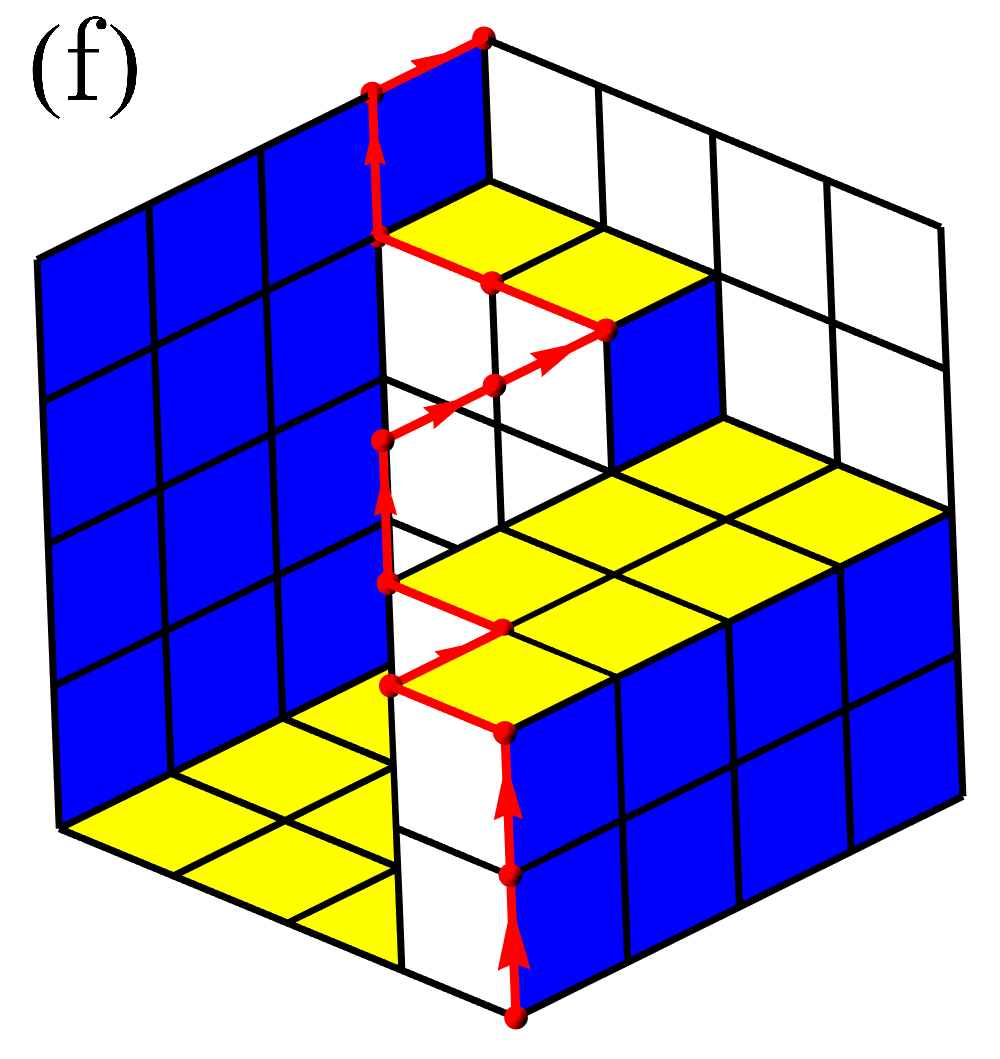}
\caption{Six lattice paths between $\mathbf{0}$ and $(4,4,4)$ in the lattice $[4 \times 4 \times 4]$. Each path encloses the same area, and the area is the same partition of $20=2+4+4+10$. The degeneracy in this case is therefore six.}
% \par\end{centering}
    \label{fig:main-22}
\end{figure}

\subsection{Distribution of $k$-hop path counts}\label{sec:results}

Consider $m_1,\dots,m_{k-1} \sim \textrm{Poisson}(\lambda|L|)$, where $|L|$ is given by Eq.~\eqref{e:width}.   We denote by $\sigma_{k}(0,1)=\sigma_{k}$ the number of $k$-hop paths between the vertices $0,1$ in the vertex set of the graph $G(V,E)$. In this section, for $n \leq 0$, $k \geq 2$, we compute the p.m.f. $P(\sigma_{k}=n)$ for all $k$. 
We now detail the proof of Theorem \ref{t:p1-2}, with a short introduction to the intersecting intervals known as \textit{lenses}.
\subsubsection{Lenses}
\begin{figure}
\centering
\includegraphics[scale=0.265]{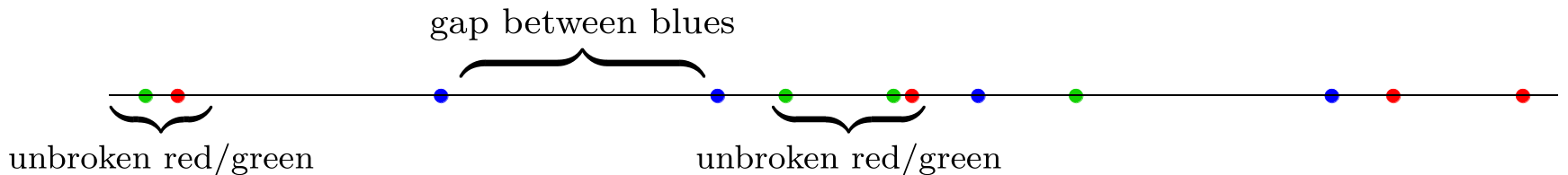}

\caption{The overlapped lenses from Fig. \ref{fig:used}. The are five ``gaps'' between the blue balls, in which lie unbroken sequences of red and green balls.}
\label{fig:k51}
\vspace{4mm}
\includegraphics[scale=0.265]{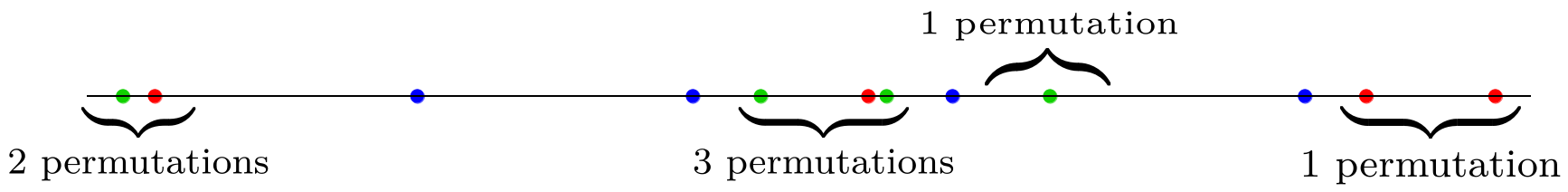}

\caption{The permutation of Fig. \ref{fig:used}, but with two of the nodes in the 3rd blue gap (from the left) swapped. This gives a new lattice path, but not a new integer partition of the number of paths. The product of the number of ways of permuting the red and green nodes, within each a blue gap, leads to the partition degeneracy of Eq.~\eqref{e:prodbin1}.}
\label{fig:k52}
\end{figure}

Since in the case $r_0 \leq 1/k$ we necessarily have $\sigma_{k}=0$, in what follows, we assume that $r_0 > 1/ k$.
With $B(a,b)$ the ball centered at $a$ of radius $b$, we define disjoint intervals $L_1  < \cdots < L_{k-1}$,
\begin{equation}
L_j := B(0,jr_0) \cap B(1,(k-j)r_0), \qquad j = 1,\ldots , d, 
\end{equation}
called ``lenses'', of same length
\begin{equation}\label{e:width}
  |L_i|=kr_0 - 1>0,
\end{equation}
where, using $\#A$ for the cardinality of a set $A$, the number of points in the intersection of the lens and the points $P_{\lambda}$,
\begin{equation}
 \#\{\mu: \mu \in L_j \cap P_{\lambda}\} = m_j.
\end{equation}
 Importantly, we also have an upper bound on the connection range $r_0$, where we therefore restrict
\begin{equation}
\frac{1}{k} < r_0 < \frac{1}{k-1}
\end{equation}
to stop the lenses overlapping.
We therefore consider i.i.d. random variables $m_1,\dots,m_{k-1} \sim \textrm{Poisson}(\lambda|L|)$.

\subsubsection{Overlapping lenses}\label{sec:overlap}
The lenses will overlap for large enough values of $r_0$ for each choice of $k$, not considering the restriction of Eq. \ref{e:range1}, but since the paths will still hop between the lenses sequentially, this does not affect the combinatorial idea of lattice path enumeration determining the $k$-hop path count $\sigma_{k}$. See e.g. Fig \ref{fig:3hops} for a depiction of the non-overlapping case for $k=3$, and Fig. \ref{fig:3hops2} for a depiction in the case $k=4$. We leave the complete solution to all ranges of $r_0$ to a later work, focusing on showing as clearly as possible this link with lattice path combinatorics, and integer partitions.

\subsubsection{Proof}

\begin{proof}[Proof of Theorem \ref{t:p1-2}]
  When $k=1$, then $\sigma_{1}=\mathbf{1}_{\{r_{0}\geq1\}}$. When $k=2$ there exists a two-hop path if and only if $L_1$ contains $m_1>0$ vertices of $\mathcal{X}$, and $\sigma_{2} \sim \text{Poisson}(2r_0 - 1)$. Now consider $k=3$. This is depicted in Fig. \ref{fig:3hops}. By overlapping the lenses and considering the relative locations of the nodes
in $L_1$ and $L_2$ in a new lens of unit width, we represent the locations of nodes $\{1,\ldots, m_1\}$ in $[0,|L_{1}|]$, or of the nodes $\{m_1,\ldots , m_1+m_2\}$ in $[0,|L_{2}|]$ using a sequence $\{Y_{i}\}_{i=1}^{m_{1}+m_{2}}$, with $ Y_i \sim \text{Uniform}[0,1]$, then letting
\begin{equation}
 X_i := \begin{cases}
1 & \text{for  } 1 \leq i \leq m_1 \\
2 & \text{for  } m_1 + 1 \leq i \leq m_1 + m_2,
  \end{cases}
\end{equation}
 then $\sigma_3$ may be written 
 \begin{equation}\label{e:janson}
  \sigma_{3} =  \sum_{j=1}^{m_2}\sum_{i=1}^{m_1}\mathbf{1}_{\{X_{i} > X_{j}\}} \mathbf{1}_{\{Y_{i} < Y_{j}\}}
\end{equation} 
Since the $m_1 +m_2$ nodes are in a specific permutation given by their relative locations $Y_i,1 \leq i \leq m_1+m_2$, we observe that $\sigma_{3}$ can be written as the area under a 2d lattice path $\Pi_{2} \in \mathcal{M}_2$ of $m_1 + m_2$ steps. For example, with $e_1=(1,0)$ a right turn, and $e_2=(0,1)$ an up turn, Fig. \ref{fig:3hops} depicts the case
\begin{equation}\label{e:ex1}
  \Pi_{k-2} =  (e_2,e_1,e_1,e_2,e_1,e_2,e_1,e_2)
  \end{equation}
where we are reading the steps from the right end of $L_2$, to its left end. The volume this path encloses is the Ferrers diagram of the  integer partition $\pi_{1,2}=(1,1,2,3)$, so $\pi_{1,2} \vdash \sigma_{3}$, and $\sigma_{3}=1+2+3+3=7$.

Notice that each lattice path $\Pi_{2} \in \mathcal{P}(\mathcal{M}_{2})$ occurs uniformly at random, and so we have
\begin{equation}
\mathbb{P}(\sigma_{3}=n) = \frac{1}{ {m_1 +  m_2 \choose m_1 }} \# \{\Pi_{2} \in \mathcal{P}(\mathcal{M}_{2}) : \pi_{1,2} \cdot S(\Pi_{2}) = n \}
\end{equation}
which implies Eq.~\eqref{e:3hopsprop}.

For each case $k \geq 2$, we have a sequence of integers $S(\Pi_{k-1})$ out of which we select $m_{k-1}$ parts to from a partition $\pi_{k-2,k-1}^{\star} \cdot S(\Pi_{k-2}) \vdash \sigma_{k}$. For $n \geq 0$, the more partitions $\pi_{k-2,k-1}^{\star} \cdot S(\Pi_{k-2}) \vdash n$ restricted to the set $S$, and the more sets $S$ which can provide at least one integer partition of $n$, the greater $\mathbb{P}(\sigma_{k}=n)$.
  An example for the case $k=4$ is given in Section \ref{sec:summary}.

  The main difference from the case $k \leq 3$ is the idea of \textit{partition degeneracy}. This is depicted in Figs. \ref{fig:3hops2}, \ref{fig:used} and \ref{fig:k4-1}, as well as in Figs. \ref{fig:k51} and \ref{fig:k52}. This is the following observation. A given integer partition $\pi_{k-2,k-1}^{\star} \cdot S(\Pi_{k-2}) \vdash \sigma_{k}$ is determined by a fixed set of sub-partitions $\pi_{1,2},\pi_{2,3},\dots,\pi_{k-2,k-1}$, but is insensitive to the other details of $\pi_{1,3},\pi_{1,4},\dots$ i.e. where the subscripts are not adjacent integers. Altering the details of $\pi_{1,3},\pi_{1,4},\dots$ etc. will modify only the corresponding lattice path $\Pi_{k-1}$, but not the enclosed volume $V(\Pi_{k-1})$, nor the partition $\pi_{k-2,k-1} \cdot S(\Pi_{k-2}) \vdash \sigma_{k}$. As such, in the case $k \geq 4$, each integer partition of $\sigma_{k}$ may have more than one lattice path to which it corresponds. The number of different lattice paths corresponding to a specific integer partition $\pi_{k-2,k-1} \cdot S(\Pi_{k-2}) \vdash \sigma_{k}$ is the integer-valued partition degeneracy
 \begin{equation}\label{e:prodbin1}
 \text{Degeneracy}(\pi_{k-2,k-1}^{\star} \cdot S(\Pi_{k-2}) \vdash \sigma_{k}) = \prod_{t=0}^{m_{k-2}} {\pi_{k-2,k-1}^\star (t)+\sum_{l=1}^{k-3} \pi_{k-2,l}^\star (t) \choose \pi_{k-2,k-1}^\star (t)}.
 \end{equation}
  In the example in Fig. \ref{fig:main-22} there are six paths (a)-(f), and so $\text{Degeneracy}(\pi_{2,3}^{\star} \cdot S(\Pi_{3}) \vdash \sigma_{4}) = \text{Degeneracy}(\{0,1,2,0,1\} \cdot \{0,2,4,7,10\} \vdash 20)=\text{Degeneracy}(2+4+4+10 \vdash 20) = 6$. See also Figs. \ref{fig:large1} and  \ref{fig:large2}.
Eq.~\eqref{e:prodbin1} is then summed over all partitions $\pi_{k-2,k-1}^{\star} \cdot S(\Pi_{k-2}) \vdash n$ to give the theorem.
\end{proof}

\begin{figure}
  \centering \includegraphics[scale=0.6]{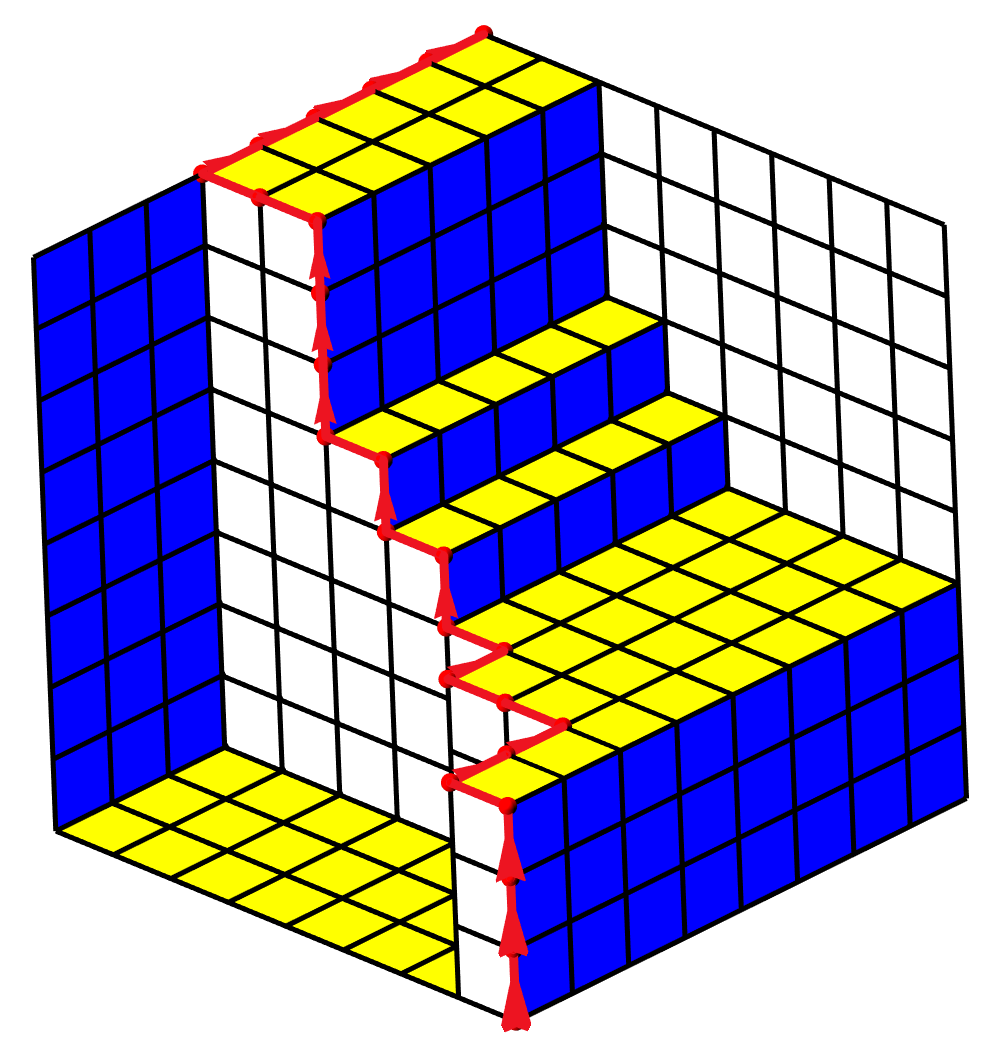}
  \caption{Partition degeneracy: (a) This partition $3+3+9+27+27+27+27+27 \vdash 150$ has a partition degeneracy of zero, since one cannot alter the red lattice path without altering the exact arrangement of the cubes beneath. It has no right turns followed by up turns, and vice versa, nor does the path walk on the walls of the cube.}\label{fig:large1}
  \includegraphics[scale=0.6]{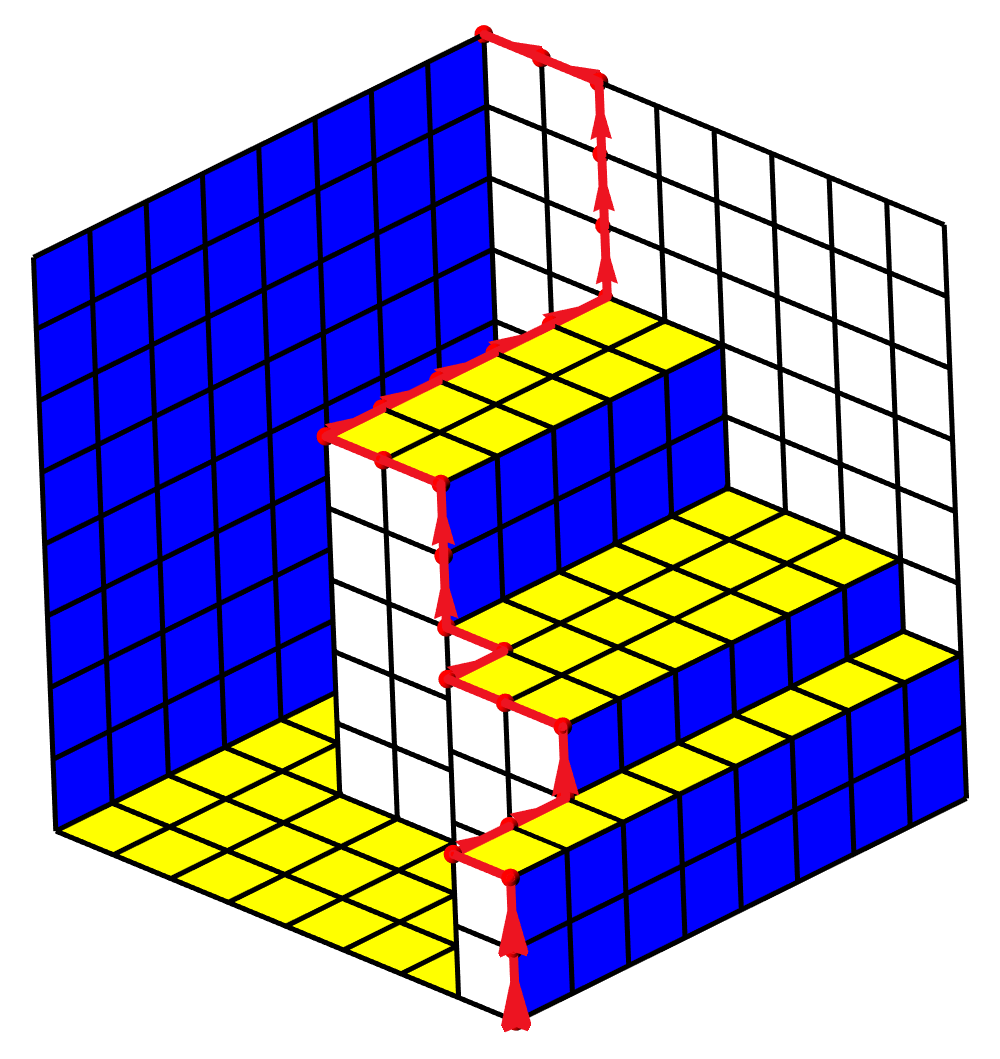}
\caption{(b) We depict $2+2+8+21+21+21+21+21 \vdash 117$. The partition degeneracy is this case is 13. Walking on the walls of $\mathcal{M}_3$ corresponds to a partition degeneracy, as can seen by the 10 different choices of routes in the final 5 steps on the red lattice path. This corresponds to the placement of balls in bins with label zero, which may be altered without affecting the overall partition. Counting these carefully may lead to a formula for the probability of zero paths.}
% \par\end{centering}
\label{fig:large2}
\end{figure}

\subsection{Probability generating function of $k$-hop path counts}\label{sec:gen}

\begin{figure}
\centering

\includegraphics[scale=0.3]{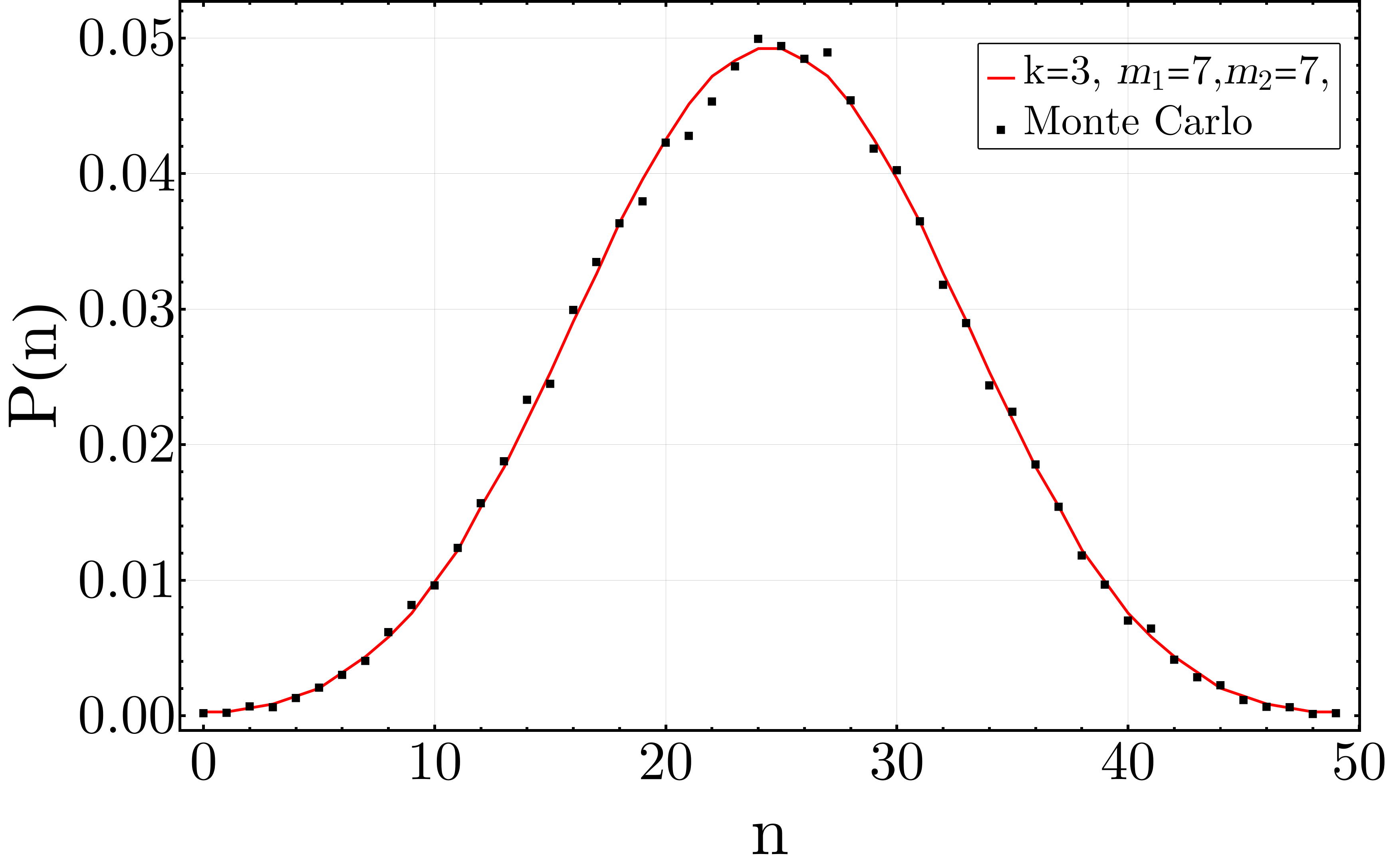}

\caption{Monte Carlo simulations (black dots), compared with our analytical calculations of Eq. \ref{e:eq1} (red line). We take $k=3$, and with two lenses in this case, we have $m_1,m_2=7$.}
\label{fig:k81}
\end{figure}

\begin{figure}
\centering
\includegraphics[scale=0.3]{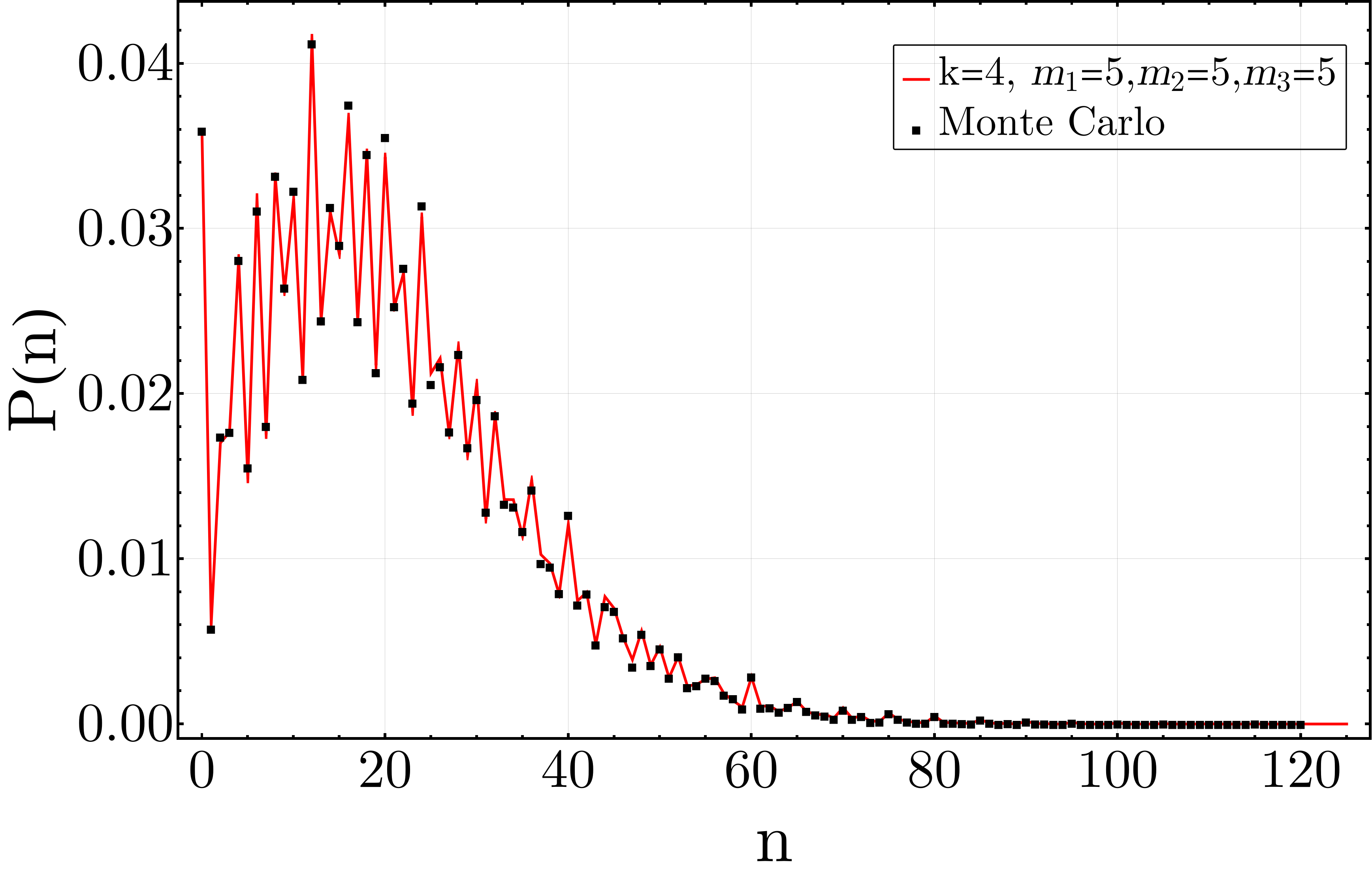}
\caption{Monte Carlo simulations (black dots), compared with our analytical calculations of Eq. \ref{e:eq1} (red line). We take $k=4$, and with three lenses in this case, we have $m_1,m_2,m_3=5$.}
\label{fig:k82}
\end{figure}

\begin{proof}[Proof of Theorem \ref{t:p1-4}] 
Recall that when $k=1$, then $\sigma_{1}=\mathbf{1}_{\{r_{0}\geq1\}}$, and when $k=2$ there exists a two-hop path if and only if $L_1$ contains $m_1>0$ vertices of $\mathcal{X}$, and $\sigma_{2} \sim \text{Poisson}(2r_0 - 1)$. The pg.f. for the trivial case $1 \leq k \leq 2$ then follows straightforwardly. When $k=3$, we have $S(\Pi_{2})=\{0,1,\dots,m_1\}$, then Eq. \eqref{e:3hopsprop} implies, via the generating function Eq. \eqref{e:fmain}, that
  \begin{equation}
 \mathbb{E}[q^{\sigma_{3}}] = \frac{1}{{m_1 + m_2 \choose m_1}} [u^{m_2}]\left( \frac{1}{\left(1-u\right)\left(1-u q\right) \cdots \left(1-u q^{m_1}\right)}\right)
  \end{equation}
  gives the p.g.f. of the number of integer partitions into $m_2$ parts selected from the set $S(\Pi_{2})$.

Using the series expansion 
\begin{equation} 
  \nonumber
  \frac{1}{(1-q)^{1+l}} = \sum_{n=0}^\infty {n+l \choose n} q^n,
 \qquad l \geq 0,  
\end{equation} 
we find that the generating function of Eq. \eqref{e:final} may be written
\begin{multline}\label{e:l1}
 \sum_{\Pi \in \mathcal{P}\left(\mathcal{M}_{k-2}\right)}  \prod_{t=0}^{m_{k-2}}\frac{1}{( 1 - u q^{S_{t}(\Pi)} )^{1+\sum_{l=1}^{k-3}\pi_{k-2,l}^\star (t)} }
 \\=   \sum_{\Pi \in \mathcal{P}\left(\mathcal{M}_{k-2}\right)} \prod_{t=0}^{m_{k-2}} \sum_{n=0}^\infty {n+\sum_{l=1}^{k-3}\pi_{k-2,l}^\star (t) \choose n} u^{n} q^{n S_t(\Pi )}
\end{multline}
Due to the technique used when constructing a basic, bivariate generating function for restricted integer partitions, we straightforwardly have that the r.h.s. of Eq. \eqref{e:l1} may be written
\begin{multline}
 \sum_{\Pi \in \mathcal{P}\left(\mathcal{M}_{k-2}\right)} \prod_{t=0}^{m_{k-2}} \sum_{n=0}^\infty {n+\sum_{l=1}^{k-3}\pi_{k-2,l}^\star (t) \choose n} u^{n} q^{n S_t(\Pi )}   \\ = \sum_{m_{k-1}=0}^\infty  \sum_{\nu=0}^\infty \left[ \sum_{\Pi \in \mathcal{P}( [ m_{k-2} \times m_{k-1} ]) \atop S(\Pi_{k-2}) \cdot \pi^\star(\Pi) = \nu  }\prod_{t=0}^{m_{k-2}}   {\pi^\star (t) +\sum_{l=1}^{k-3}\pi_{k-2,l}^\star (t) \choose \pi^\star (t)} \right] q^{\nu}u^{m_{k-1}},  
\end{multline}
and so the univariate generating function of Eq. \eqref{e:final} may be written as the following power series,
\begin{multline}
 \frac{ 1}{{m_1 + \cdots +  m_{k-1} \choose m_1,\ldots, m_{k-1} }}
         [u^{m_{k-1}}] \sum_{\Pi \in \mathcal{P}\left(\mathcal{M}_{k-2}\right)}\prod_{t=0}^{m_{k-2}}\frac{1}{( 1 - u q^{S_{t}(\Pi)} )^{1+\sum_{l=1}^{k-3}\pi_{k-2,l}^\star (t)} }
\\= \frac{ 1}{{m_1 + \cdots +  m_{k-1} \choose m_1,\ldots, m_{k-1} }}
        \sum_{n=0}^\infty \left[ \sum_{\Pi \in \mathcal{P}\left(\mathcal{M}_{k-2}\right)}  \sum_{\pi \in {\cal P}( [ m_{k-2} \times m_{k-1} ] )  \atop  S(\Pi) \cdot \pi^\star = n  }\prod_{t=0}^{m_{k-2}}{\pi^\star (t) +\sum_{l=1}^{k-3}\pi_{k-2,l}^\star (t) \choose \pi^\star (t)} \right]q^n, 
\end{multline} 
which provides the conclusion.
\end{proof}

\subsubsection*{Acknowledgements}

We thank Marc Barthelemy, Ginestra Bianconi, Carl P. Dettmann, Orestis Georgiou, Suhanya Jayaprakasam, Jon Keating, Sunwoo Kim, Georgie Knight and Dusit Niyato for many helpful discussions. This research is supported by the Ministry of Education, Singapore, under its AcRF Tier 1 grant MOE2018-T1-001-201 RG25/18. The first author was also partially supported by the EPSRC grant \textit{Random Walks on Random Networks}, Institutional Sponsorship 2015, and thanks J\"{u}rgen Jost for kind hospitality at the Max Plank Institute for Mathematics in the Sciences, Leipzig, during part of this study in 2018. We would also like to thank Szabolcs Horv\'{a}t for his development of the IGraphM package in Mathematica \cite{horvatcode}, used throughout this research.

%\bibliographystyle{IEEEtran}
%\bibliography{refs}


\begin{appendices}
\section{Appendix A: Code for Monte Carlo corroboration}\label{a:1}
Here we contain the code used to corroborate Eq.~\eqref{e:eq1}, for the case $k=3$ and $k=4$. The output of this code is the two graphs in Figs. \ref{fig:k81} and \ref{fig:k82}
\begin{lstlisting}[language=Mathematica,caption={Example code}]
makeedges = 
 Function[{subsets, r0}, 
  Select[subsets, 
   Abs[#[[1]] - #[[2]]] < r0 &]];(*Make the edges of the graph*)
graph[vert_, r0_] := 
 Graph[vert, 
  UndirectedEdge @@@ 
   makeedges[Subsets[vert, {2}], 
    r0]];(*Make the 1d random geometric graph*)
makegraph[nv_, coord_, width_, height_, r0_] := 
 Module[{pts, newvertices, newedges, edges, alledges, allpts, e1, e2, 
   ew}, allpts = Table[RandomReal[{-width/2, width/2}], {i, 1, nv}];
  allpts = Join[allpts, coord];
  pdg1 = graph[allpts, r0]];
Clear[r0];
domainwidth = 1;
mk3hops[n_, r0_, k_] := Module[{gr, pthcount, m1, m2},
   m1 = 7;
   m2 = 7;
   pt1 = RandomReal[{0.5 - 2 r0, r0 - 0.5}, m1];
   pt2 = RandomReal[{0.5 - r0, 2 r0 - 0.5}, m2];
   ptall = Join[pt1, pt2, {-0.5, 0.5}];
   gr = makegraph[0, ptall, domainwidth, 0.002, r0];
   pthcount = Length[FindPath[gr, -0.5, 0.5, {k}, All]];
   pthcount];
mk4hops[n_, r0_, k_] := Module[{gr, pthcount, m1, m2, m3},
   m1 = 5;
   m2 = 5;
   m3 = 5;
   pt1 = RandomReal[{0.5 - 3 r0, r0 - 0.5}, m1];
   pt2 = RandomReal[{0.5 - 2 r0, 2 r0 - 0.5}, m2];
   pt3 = RandomReal[{0.5 - r0, 3 r0 - 0.5}, m3];
   ptall = Join[pt1, pt2, pt3, {-0.5, 0.5}];
   gr = makegraph[0, ptall, domainwidth, 0.002, r0];
   pthcount = Length[FindPath[gr, -0.5, 0.5, {k}, All]];
   pthcount];
data1 = Table[mk3hops[0, 0.35, 3], {i, 1, 50000}];
d3 = HistogramList[data1, {1}, "Probability"][[2, All]];
d4 = Transpose[{Range[0, Length@d3 - 1], d3}];
ser1 = Series[QBinomial[14, 7, q], {q, 0, 49}];
cf1 = CoefficientList[ser1, q];
d1 = Transpose[{Range[0, Length@cf1 - 1], 
    N@(1/Binomial[14, 7]) cf1}];
ListPlot[{d1, d4}, 
 Joined -> {True, 
   False}](**Plot a graph of the coefficients against the Monte Carlo \
simulation for k=3**)

data2 = Table[mk4hops[0, 0.27, 4], {i, 1, 50000}];
bincnts = (Join[{0}, PositionIndex[#][0], {Length[#] + 1}] // 
    Rest[#] - 1 - 
      Most[#] &) &;(*Count balls in bins,to get the part counts of \
the complementary partition*)
nt[x_] := Module[{a}, If[x == 1, a = 0;]; If[x == 0, a = 1;]; a]
tf[x_] := Module[{a, li}, a = bincnts@x;
  li = {};
  For[i = 1, i <= Length@a, i++, 
   For[j = 1, j <= a[[i]], j++, li = Append[li, i - 1];];];
  {Prepend[Accumulate@li, 0], bincnts[nt[#] & /@ x]}]
g[x_] := Product[
  1/(1 - u q^tf[x][[1, t + 1]])^(1 + tf[x][[2, t + 1]]), {t, 0, 5}]
d7 = Transpose[{Table[i, {i, 0, 125}], 
    1/Multinomial[5, 5, 5] CoefficientList[
      CoefficientList[
        Series[Total[
          g[#] & /@ Permutations[{0, 0, 0, 0, 0, 1, 1, 1, 1, 1}]], {u,
           0, 5}], u][[6]], q]}];
d5 = HistogramList[data2, {1}, "Probability"][[2, All]];
d6 = Transpose[{Range[0, Length@d5 - 1], d5}];
ListPlot[{d7, d6}, 
 Joined -> {True, 
   False}](**Plot a graph of the coefficients against the Monte Carlo \
simulation for k=4**)
\end{lstlisting}
\end{appendices}
\end{document}